\documentclass[11pt]{amsart}
\usepackage{latexsym}
\usepackage{amssymb}
\usepackage{amsmath}
\usepackage{amsthm}
\usepackage{enumerate}
\usepackage{bbm}
\usepackage{verbatim}

\newtheorem{thm}{Theorem}[section]
\newtheorem{lem}[thm]{Lemma}

\newtheorem{prop}[thm]{Proposition}

\theoremstyle{definition}

\newtheorem{exmp}[thm]{Example}

\theoremstyle{remark}
\newtheorem{rem}[thm]{Remark}

\numberwithin{equation}{section}

\def\1{\mathbbm{1}}
\renewcommand{\qed}{\unskip\nobreak\quad\qedsymbol}

\newcommand{\beq}{\begin{equation}}
\newcommand{\eeq}{\end{equation}}

\allowdisplaybreaks

\begin{document}
\title{Absolute and Unconditional Convergence of Series \linebreak of Ergodic Averages and Lebesgue Derivatives}

\author{Bryan Johnson}
\address{Alaska Department of Transportation\\
Fairbanks, Alaska}
\email{bryan.johnson@alaska.gov}

\author{Joseph Rosenblatt}
\address{Dept of Mathematics\\
University of Illinois at Urbana-Champaign\\
1409 W. Green St\\
Urbana, IL 61801}
\email{rosnbltt@illinois.edu}

\subjclass[2020]{Primary: 28D05, 42A61; Secondary:  37A46}
\copyrightinfo{2020}{AMS}

\date{\today}

\begin{abstract}  We consider when there is absolute or unconditional convergence of series of various types of stochastic processes.  These processes include differences of  averages in ergodic theory and harmonic analysis, like the classical Ces\`aro average in ergodic theory and Lebesgue derivatives in harmonic analysis.   
\end{abstract}
\maketitle

\section{Introduction}\label{1}

Unconditional convergence of a series in a Banach space has a number of different equivalent formulations. We are interested specifically in series of the form $\sum\limits_{k=1}^{\infty} M_{n_{k+1}} f-M_{n_{k}} f$ where $\left(n_{k}\right)$ is a non-decreasing sequence of whole numbers, and $M_{n}$ denotes an average, like the usual average in ergodic theory. It turns out that unconditional convergence of such differences depends on spectral properties of the averages. Also of interest in this paper will be series of the form $\sum\limits_{k=1}^{\infty}D_{\epsilon_{k+1}} f-D_{\epsilon_{k}} f$ where $D_{\epsilon} f$ is defined to be the average of $f$ over a small interval.

Section~\ref{2} contains some basic definitions and characterizations of unconditional convergence, weak unconditional convergence, and absolute convergence that will be used. It will be shown that for any sequence $\left(\epsilon_{k}\right)$ going to zero, absolute convergence of the series of operators $\sum\limits_{k=1}^\infty D_{\epsilon_{k+1}}f -D_{\epsilon_{k}}f$ fails in general whether viewed as operators on $L^2(\mathbb{R})$ or on $L^2(\mathbb{Z})$. Similarly, it will also be shown that for any sequence of integers $\left(n_{k}\right)$ going to $\infty$, absolute convergence of the series of operators $\sum\limits_{k=1}^{\infty} M_{n_{k+1}}f - M_{n_{k}}f$ also fails in general. Furthermore, it is shown that for any sequence of integers $\left(n_{k}\right)$ going to infinity, there is an $f$ in $L^1(X)$ such that weak unconditional convergence of $\sum\limits_{k=1}^{\infty} M_{n_{k+1}} f-M_{n_{k}} f$ fails in $L^1(X)$.

Section~\ref{3} contains an interesting spectral characterization for determining when $\sum\limits_{k=1}^{\infty} M_{n_{k+1}} f-M_{n_{k}} f$ is unconditionally convergent for all $f$ in $L^2(X)$. It also contains examples of when this characterization holds and when it fails.

Section~\ref{4} is about sequences of differences using probability measures. It is noted that there is also a spectral characterization like the one in Section~\ref{3} for determining unconditional convergence in this situation. Among other things, it is shown here that if the spectrum of $\mu$ is contained in a Stolz region, then for all $f$ in $L^2(X)$, unconditional convergence of $\sum\limits_{k=1}^{\infty} \mu^{n_{k+1}} f-\mu^{n_{k}} f$ occurs for all sequences $\left(n_{k}\right)$ going to infinity. If the measure $\mu$ is only assumed to be strictly aperiodic however, it turns out that we are only guaranteed unconditional convergence for certain sequences $\left(n_{k}\right)$ going to infinity. There is also a result which shows that if $\mu$ is only assumed to be strictly aperiodic, and $\left(w_{k}\right)$ is any sequence of reals going to infinity, then unconditional convergence of the weighted differences $\sum\limits_{k=1}^{\infty}w_k (\mu^{n_{k+1}} f-\mu^{n_{k}} f)$ can fail when $2 n_{k}{ }^{p} \geq n_{k+1} \geq n_{k}{ }^{p}$ for $p>1$.

Section~\ref{5} shows that the spectral characterizations for unconditional convergence of $\sum\limits_{k=1}^{\infty} D_{\epsilon_{k+1}} f-D_{\epsilon_{k}} f$ for $f$ in $L^2(\mathbb{T}), L^2(\mathbb{R})$ and $H_{2}(\mathbb{T})$ are all the same. Also presented here are examples comparable to those given in Section~\ref{2} for which this spectral characterization holds and fails. Lastly, it is noted that if $E_{\epsilon} f$ is defined to be $(D_{\epsilon}\circ D_{\epsilon})f$, then unconditional convergence of $\sum\limits_{k=1}^{\infty} E_{\epsilon_{k+1}} f-E_{\epsilon_{k}} f$ for $f$ in $L^2(\mathbb{T})$ will occur for all sequences of $\epsilon_{k}$ going to zero.

Section~\ref{6} uses interpolation to show that in the case of $\left(n_{k}\right)$ being lacunary, not only is $\sum\limits_{k=1}^{\infty} M_{n_{k+1}} f-M_{n_{k}} f$ unconditionally convergent in $L^p$-norm for all $f\in L^p(X)$ for $p=2$, as is shown in Section~\ref{2}, but this also holds  for all $p, 1<p<\infty$. In Section~\ref{1}, the case of $p=1$ was already shown to never hold for any increasing $n_k$.  It is also shown that for any sequence of positive numbers $\left(\epsilon_{k}\right)$ satisfying $\epsilon_{k+1} \leq r \epsilon_{k}$ where $r<1$ and for any $1<p<\infty$ then for any $f\in L^p$, the series $\sum\limits_{k=1}^{\infty} D_{\epsilon_{k+1}} f-D_{\epsilon_{k}} f$ is unconditionally convergent in $L^p$-norm, extending the $p=2$ version of this which was proved in Section~\ref{4}.

Section~\ref{7} examines what happens when the series have terms like the ergodic averages themselves, in general and for specific dynamical systems.  Now the absolute and unconditional convergence is typically only a first Baire category phenomenon.  But it is not so clear when unconditional convergence can happen with absolute convergence failing.

Section~\ref{8} gives a list a few of the questions that have come up when pursuing the research in this article.  These questions, and their answers, might bring into sharper focus what some of the issues are that come up when considering the results given here.

\section{Unconditional Convergence}\label{2}
A series $\sum\limits_{k=1}^{\infty} f_{k}$ in a Banach space $X$ is said to be unconditionally convergent if for all permutations $\pi$ of $\{1,2,3, \ldots\}$ the series $\sum\limits_{k=1}^{\infty} f_{\pi(k)}$ converges in the norm topology. The series is weakly unconditionally convergent if for all elements $f^{*}$ in the dual space of $X$, the series $\sum\limits_{k=1}^{\infty} f^{*}\left(f_{k}\right)$ converges unconditionally. We have the following two theorems. See Lindenstrauss and Tzafriri [10] and Day [7] for the background information.

\begin{thm}\label{2.1}  For a series $\sum\limits_{k=1}^{\infty} f_{k}$ in a Banach space, the following are equivalent:

(a) $\sum\limits_{k=1}^{\infty} f_{k}$ is unconditionally convergent,

(b) $\quad \sum\limits_{k=1}^{\infty} c_{k} f_{k}$ converges for all sequences $\left(c_{k}\right),\left|c_{k}\right| \leq 1$ for all $k=1,2,3, \ldots$,

(c) $\quad \lim _{m, n \rightarrow \infty} \sup _{\left|c_{k}\right| \leq 1}\left\|\sum\limits_{k=m}^{n} c_{k} f_{k}\right\|=0$.
\end{thm}

\noindent {\bf Remark}: As a consequence of this, if $\sum\limits_{k=1}^{\infty} f_{k}$ is unconditionally convergent, then we would have that $\sup _{M \geq 1} \sup _{\left|c_{k}\right| \leq 1}\left\|\sum\limits_{k=1}^{M} c_{k} f_{k}\right\|$ is bounded and $\sup _{\left|c_{k}\right| \leq 1}\left\|\sum\limits_{k=1}^{\infty} c_{k} f_{k}\right\|$ is bounded. \qed

\begin{thm}\label{2.2}  For a series $\sum\limits_{k=1}^{\infty} f_{k}$ in a Banach space, the following are equivalent:

(a) $\sum\limits_{k=1}^{\infty} f_{k}$ is weakly unconditionally convergent,

(b) $\quad \sup _{M \geq 1} \sup _{\left|c_{k}\right| \leq 1}\left\|\sum\limits_{k=1}^{M} c_{k} f_{k}\right\|$ is bounded.
\end{thm}

Generally, weakly unconditionally convergent series are not unconditionally convergent. However, the two properties are equivalent in reflexive Banach spaces. In particular, for the reflexive spaces $L^p(X, m), 1<p<\infty$, they are equivalent. See for example Theorem 5 in Bessaga and Pelczy\'nski~\cite{BP}.  Hence, we have the following result.

\begin{thm}\label{2.3} If $\sum\limits_{k=1}^{\infty} f_{k}$ is a series in $L^p(X, m), 1<p<\infty$, then the following are equivalent:

(a) $\sum\limits_{k=1}^{\infty} f_{k}$ is unconditionally convergent,

(b) $\sup _{M \geq 1} \sup _{\left|c_{k}\right| \leq 1}\left\|\sum\limits_{k=1}^{M} c_{k} f_{k}\right\|_{p}$ is bounded.
\end{thm}

Suppose that $\left(T_{k}\right)$ is a sequence of bounded linear operators on $L^p(X, m)$. Then the above gives this homogeneous version of unconditional convergence.

\begin{thm}\label{2.4}  The following are equivalent when $1<p<\infty$ :

(a) for all $f \in L^p(X, m), \sum\limits_{k=1}^{\infty} T_{k} f$ is unconditionally convergent,

(b) there is a constant $C$ such that for all $f \in L^p(X, m)$,

$$
\sup _{M \geq 1,\left|c_{k}\right| \leq 1}\left\|\sum\limits_{k=1}^{M} c_{k} T_{k} f\right\|_{p} \leq C\|f\|_{p}
$$
\end{thm}
\begin{proof}
Clearly, for a fixed $f \in L^p(X, m)$, (b) implies (b) in Theorem~\ref{2.3}, and so $\sum\limits_{k=1}^{\infty} T_{k} f$ is unconditionally convergent by Theorem~\ref{2.3}. Conversely, assume (a). Then for each $f$, $\left\|\sum\limits_{k=1}^{M} c_{k} T_{k} f\right\|_{p}$ is bounded independently of $M$ and $\left(c_{k}\right),\left|c_{k}\right| \leq 1$. Consider the set

$$
B_{L}=\left\{f \in L^p: \sup _{M \geq 1,\left|c_{k}\right| \leq 1}  \left\|\sum\limits_{k=1}^M c_k T_{k} f\right \|_{p} \leq L\right\}
$$

\noindent This is a closed set in the norm topology because it is the intersection over all $M$ and $\left(c_{k}\right),\left|c_{k}\right| \leq 1$ of the sets

$$
B_{L}\left(M,\left(c_{k}\right)\right)=\left\{f \in L^p:\left\|\sum\limits_{k=1}^{M} c_{k} T_{k} f\right\|_{p} \leq L\right\}
$$

\noindent and each of these sets is closed in the norm topology. So by the Baire Category Theorem, there exists some $L_{0}$ such that $B_{L_{0}}$ has non-empty interior. That is, for some $f_{0} \in L^p$, some $\delta_{0}>0$, if $\|f\|_{p} \leq \delta_{0}$, then $f+f_{0} \in B_{L_{0}}$. Hence, if $\|f\|_{p} \leq \delta_{0}$, then for all $M \geq 1$ and $\left(c_{k}\right),\left|c_{k}\right| \leq 1$,

$$
\begin{aligned}
\left\|\sum\limits_{k=1}^{M} c_{k} T_{k} f\right\|_{p} & \leq\left\|\sum\limits_{k=1}^{M} c_{k} T_{k} f_{0}\right\|_{p}+\left\|\sum\limits_{k=1}^{M} c_{k} T_{k}\left(f+f_{0}\right)\right\|_{p} \\
& \leq 2 L_{0} .
\end{aligned}
$$

Hence, if $0 \neq f \in L^p$, then for $M \geq 1$ and $\left(c_{k}\right),\left|c_{k}\right| \leq 1$,

$$
\left\|\sum\limits_{k=1}^{M} c_{k} T_{k}\left(\frac{f \delta_{0}}{\|f\|_{p}}\right)\right\|_{p} \leq 2 L_{0}
$$

That is,

$$
\left\|\sum\limits_{k=1}^{M} c_{k} T_{k} f\right\|_{p} \leq\left(\frac{2 L_{0}}{\delta_{0}}\right)\|f\|_{p}
$$
\hfill  \end{proof}

A series $\sum\limits_{k=1}^{\infty} x_{k}$ in a Banach space $X$ is said to be absolutely convergent if $\sum\limits_{k=1}^{\infty}\left\|x_{k}\right\|<\infty$ It is an elementary fact that absolute convergence of $(x_{k})$ implies unconditional convergence.  The converse however is not necessarily true. A standard example for this is in $\ell_{2}$ where we take $x_{k}=(0, \ldots, 0,1 / k, 0, \ldots, 0)$ where $1 / k$ is in the $k^{t h}$ position It is easy to see that for any permutation $\pi, \sum\limits x_{\pi(k)}$ converges to $(1,\frac 12,\frac 13,\frac 14, \ldots)$ but yet $\sum\limits_{k=1}^{\infty}\left\|x_{k}\right\|= \sum\limits_{k=1}^{\infty} 1 / k=\infty$.   It is actually much harder to construct an example of an unconditionally convergent series that is not absolutely convergent in $\ell_1$.  In any case, the famous Dvoretzky-Rogers Theorem shows that there are examples for this in every infinite dimensional Banach space. See Heil~\cite{Heil} for details, and further references.

The ergodic averages $M_{n} f=\frac{1}{n} \sum\limits_{k=1}^{n} f\left(\tau^{k} x\right)$ and the differentiation operator $D_{\epsilon} f=$ $\varphi_{\epsilon} * f$ have been known for quite some time to behave similarly. Here $D_{\epsilon} f=\varphi_{\epsilon} * f$ is defined so that either

$$
\varphi_{\epsilon}(\theta)= \begin{cases}1 / 2 \epsilon & \text { when }-\epsilon \leq \theta \leq+\epsilon \\ 0 & \text { otherwise }\end{cases}
$$

or

$$
\varphi_{\epsilon}\left(e^{i \theta}\right)= \begin{cases}1 / 2 \epsilon & \text { when }-\epsilon \leq \theta \leq+\epsilon \\ 0 & \text { otherwise }\end{cases}
$$

\noindent depending on whether or not $f$ is defined on $\mathbb{R}$ or on $\mathbb{T}$ respectively. The analogy between $M_{n} f$ and $D_{\epsilon} f$ starts to become apparent in comparing the Ergodic Theorem with the Lebesgue Differentiation Theorem. A significant part of this paper will consist of making further comparisons between the two in the context of unconditional convergence.

The following facts about the Fourier Transform of the differentiation operator will be used extensively: if $f$ is a function in $L^2(\mathbb{R})$, then

$$
\widehat{D_{\epsilon} f}(t)=\widehat{\varphi_{\epsilon}} \widehat{f}(t)=\frac{\sin t \epsilon}{t \epsilon} \widehat{f}(t)
$$

for all $t$ in $\mathbb{R}$. Whereas, if $f$ is a function in $L^2(\mathbb{T})$, then

$$
\widehat{D_{\epsilon} f}(\ell)=\widehat{\varphi_{\epsilon}}(\ell) \widehat{f}(\ell)=\frac{\sin \ell \epsilon}{\ell \epsilon} \widehat{f}(\ell)
$$

for all $\ell$ in $\mathbb{Z}$.

When discussing averages, functions on $\mathbb{Z}$ will be always be denoted by $\varphi$ so that $M_{n} \varphi(m)$ will be unambiguously defined to be $\frac{1}{n} \sum\limits_{k=1}^{n} \varphi(m+k)$. Also, $m_{n}(\gamma)$ will be defined by $\frac{1}{n} \sum\limits_{k=1}^{n} \gamma^{k}$ for all $\gamma \in \mathbb{T}$.

In light of the relationship between absolute and unconditional convergence, it is important to compare Example~\ref{5.7}  and Example~\ref{5.8} with the following results which show that one never has absolute convergence of the differences of the differentiation operators.

\begin{thm}\label{2.5}  If $\epsilon_{k}$ is any sequence of positive reals converging to zero and if $D_{\epsilon_{k}}$ is viewed as a sequence of operators from $L^2(\mathbb{R})$ to $L^2(\mathbb{R})$ then

$$
\sum\limits_{k=1}^{\infty}\left\|D_{\epsilon_{k+1}}-D_{\epsilon_{k}}\right\|_{2}=\infty
$$
\end{thm}

\begin{thm}\label{2.6}  If $\epsilon_{k}$ is any sequence of positive reals converging to zero and if $D_{\epsilon_{k}}$ is viewed as a sequence of operators from $L^2(\mathbb{T})$ to $L^2(\mathbb{T})$ then

$$
\sum\limits_{k=1}^{\infty}\left\|D_{\epsilon_{k+1}}-D_{\epsilon_{k}}\right\|_{2}=\infty
$$
\end{thm}

Theorem~\ref{2.6} is slightly harder to prove than Theorem~\ref{2.5} and will require the use of Bernstein's Inequality (See volume II, p. 276, Theorem 7.24 of Zygmund~\cite{Zygmund}).  Bernstein's theorem is a statement about integral functions of type at most $\sigma$. The class of such functions is denoted $E^{\sigma}$. A function $F$ will be said to be in $E^{\sigma}$ if if it is entire and for all $\epsilon>0$ satisfies

$$
F(z)=O\left(e^{(\sigma+\epsilon)|z|}\right).
$$

\begin{thm}\label{2.7} Suppose $F$ is in $E^{\sigma}$, then with 
$M=\sup _{x \in \mathbb{R}}|F(x)|$, for all $x$ in $\mathbb{R}$,

$$
\left|F^{\prime}(x)\right| \leq \sigma M.
$$

\end{thm}

For example, if $h$ is in $L^{2}(-\sigma, \sigma)$, then $\hat{h}$ is in $E^{\sigma}$. This is easy to see by looking at the integral representing $\hat{h}$ :

$$
\hat{h}(z)=\int_{-\sigma}^{\sigma} h(\xi) e^{i \xi z} d \xi
$$

thus

$$
|\hat{h}(z)| \leq e^{\sigma|z|} \int_{-\sigma}^{\sigma}|h(\xi)| d \xi
$$

The following lemmas will be useful in the proof of Theorem~\ref{2.5}.

\begin{lem}\label{2.8} If $0<a<b$, then

$$
\frac{2}{\pi}(1-a / b) \leq \sup _{t>0}\left|\frac{\sin (a t)}{a t}-\frac{\sin (b t)}{b t}\right|
$$
\end{lem}
\begin{proof}

$$
\begin{aligned}
\sup _{t>0}\left|\frac{\sin (a t)}{a t}-\frac{\sin (b t)}{b t}\right| & \geq\left|\frac{\sin \left(\frac{a \pi}{b}\right)}{\frac{a \pi}{b}}-\frac{\sin \left(\frac{b \pi}{b}\right)}{\frac{b \pi}{b}}\right| \\
& =\frac{\sin \left(\frac{a \pi}{b}\right)}{\frac{a \pi}{b}}
\end{aligned}
$$

Case I: When $b / 2<a<b$ then $\pi / 2<a \pi / b<\pi$. However, we know that $\sin x>2[1-x / \pi]$ when $\pi / 2<x<\pi$, thus

$$
\frac{\sin \left(\frac{a \pi}{b}\right)}{\frac{a \pi}{b}}>\frac{2[1-a / b]}{\frac{a \pi}{b}}>\frac{2}{\pi}(1-a / b)
$$

Case II: When $0<a<b / 2$ then $0<a \pi / b<\pi / 2$. Also, we know that $\sin x>2 x / \pi$ when $0<x<\pi / 2$ Thus,

$$
\frac{\sin \left(\frac{a \pi}{b}\right)}{\frac{a \pi}{b}}>\frac{2 a / b}{a \pi / b}=\frac{2}{\pi}>\frac{2}{\pi}(1-a / b)
$$
\end{proof}

\begin{lem}\label{2.9}  If $\left(\epsilon_{k}\right)$ is a sequence of positive real numbers converging to zero and satisfying $\epsilon_{k} \geq \epsilon_{k+1}$ then

$$
\sum\limits_{k=1}^\infty \left(1-\frac{\epsilon_{k+1}}{\epsilon_{k}}\right)=\infty
$$
\end{lem}
\begin{proof}  We recall that whenever $0 \leq u_{n}<1$,

$$
\prod_{k=1}^{\infty}\left(1-u_{n}\right)>0 \Longleftrightarrow \sum\limits_{k=1}^{\infty} u_{n}<\infty
$$

Thus upon letting $u_{k}=\left(1-\frac{\epsilon_{k+1}}{\epsilon_{k}}\right)$ it can be seen that

$$
\prod_{k=1}^{\infty}\left(1-u_{n}\right)=\prod_{k=1}^{\infty} \frac{\epsilon_{k+1}}{\epsilon_{k}}=0
$$

hence the lemma is proved.
\end{proof}

Theorem~\ref{2.5} can now be proved
\begin{proof}[Proof of Theorem~\ref{2.5}]
It can be shown that $\left\|D_{\epsilon_{k+1}}-D_{\epsilon_{k}}\right\|_{2}=\sup _{t}\left|\frac{\sin t \epsilon_{k+1}}{t \epsilon_{k+1}}-\frac{\sin t \epsilon_{k}}{t \epsilon_{k}}\right|$ from which it follows that

$$
\sum\limits_{k=1}^{\infty}\left\|D_{\epsilon_{k+1}}-D_{\epsilon_{k}}\right\|_{2}=\sum\limits_{k=1}^{\infty} \sup _{t}\left|\frac{\sin t \epsilon_{k+1}}{t \epsilon_{k+1}}-\frac{\sin t \epsilon_{k}}{t \epsilon_{k}}\right| \geq \frac{2}{\pi} \sum\limits_{k=1}^{\infty}\left(1-\frac{\epsilon_{k+1}}{\epsilon_{k}}\right)=\infty
$$

The last inequality follows from Lemma~\ref{2.8}, and the final sum diverges because of Lemma~\ref{2.9}.
\end{proof}

The following additional lemma will be needed in proving Theorem~\ref{2.6}:

\begin{lem}\label{2.10}  If $0<a<b<1$, then

$$
\sup _{t>0}\left|\frac{\sin (a t)}{a t}-\frac{\sin (b t)}{b t}\right| \leq \frac{1}{1-b} \sup _{\ell \in \mathbb{Z}}\left|\frac{\sin (a \ell)}{a \ell}-\frac{\sin (b \ell)}{b \ell}\right|
$$
\end{lem}
\begin{proof}

Since $0<a<b$, it follows that $g(z)=\frac{\sin (a z)}{a z}-\frac{\sin (b z)}{b z}$ is in $E^{\sigma}$ with $\sigma=b$. Hence, by Bernstein's inequality Theorem~\ref{2.7}, we note that for every $t \in \mathbb{R}$ and every $\ell \in \mathbb{Z}$,

$$
\begin{aligned}
|g(t)| & \leq|g(t)-g(\ell)|+|g(\ell)| \leq \sup _{\theta}\left|g^{\prime}(\theta)\right|+|g(\ell)| \\
& \leq b \sup _{\{x \in \mathbb{R}\}}|g(x)|+|g(\ell)|
\end{aligned}
$$

Thus

\begin{equation*}
\sup _{\{x \in \mathbb{R}\}}|g(x)| \leq \frac{1}{1-b} \sup _{\{\ell \in \mathbb{Z}\}}|g(\ell)| \tag{2.1}
\end{equation*}

which is the desired result.
\end{proof}

Other sampling results similar to Equation~\ref{2.1} can also be found in Section 4.3.3 on page 186 of Timan~\cite{Timan} or in Cartwright~\cite{Cartwright}.

Theorem~\ref{2.6} can now be proved.
\begin{proof}[Proof of Theorem~\ref{2.6}]

This time,

$$
\left\|D_{\epsilon_{k+1}}-D_{\epsilon_{k}}\right\|_{2}=\sup _{\{\ell \in \mathbb{Z}\}}\left|\frac{\sin \ell \epsilon_{k+1}}{\ell \epsilon_{k+1}}-\frac{\sin \ell \epsilon_{k}}{\ell \epsilon_{k}}\right|
$$

Hence by Lemma~\ref{2.10} and Lemma~\ref{2.8},

$$
\begin{aligned}
\sum\limits_{k=1}^{\infty}\left\|D_{\epsilon_{k+1}}-D_{\epsilon_{k}}\right\|_{2} & =\sum\limits_{k=1}^{\infty} \sup _{\{\ell \in \mathbb{Z}\}}\left|\frac{\sin \ell \epsilon_{k+1}}{\ell \epsilon_{k+1}}-\frac{\sin \ell \epsilon_{k}}{\ell \epsilon_{k}}\right| \\
& \geq \sum\limits_{\left\{k: \epsilon_{k}<1\right\}}\left(1-\epsilon_{k}\right) \sup _{\{t \in \mathbb{R}\}}\left|\frac{\sin t \epsilon_{k+1}}{t \epsilon_{k+1}}-\frac{\sin t \epsilon_{k}}{t \epsilon_{k}}\right| \\
& \geq \frac{2}{\pi} \sum\limits_{\left\{k: \epsilon_{k}<1\right\}}\left(1-\epsilon_{k}\right)\left(1-\frac{\epsilon_{k+1}}{\epsilon_{k}}\right) \\
& =\frac{2}{\pi} \sum\limits_{\left\{k: \epsilon_{k}<1\right\}} 1-\frac{\epsilon_{k+1}}{\epsilon_{k}}+\epsilon_{k+1}-\epsilon_{k} \\
& =-\epsilon_{n_{1}}+\frac{2}{\pi} \sum\limits_{\left\{k: \epsilon_{k}<1\right\}}\left(1-\frac{\epsilon_{k+1}}{\epsilon_{k}}\right) .
\end{aligned}
$$

In the above, we denote the first $\epsilon_{k}$ which is less than one by $\epsilon_{n_{1}}$. Also, since $\lim \epsilon_{k}=0$, it follows from Lemma 2.9 that the last series is infinite.
\end{proof}

The ergodic version of Theorem~\ref{2.5} can be easily proved using a discrete version of Lemma~\ref{2.9} together with an averaging version of Lemma~\ref{2.8}. The alternate versions of Lemma~\ref{2.9} and Lemma~\ref{2.8} have nearly identical proofs to the original so they will merely be stated here for convenience without proof.  The alternate version of Theorem~\ref{2.5} will be both stated and proved.

\begin{lem}\label{2.11} If $1 \leq n_{1} \leq \ldots n_{k} \leq n_{k+1} \leq \ldots$ and $\lim _{k \rightarrow \infty} n_{k}=\infty$, then

$$
\sum\limits_{k=1}^{\infty}\left(1-\frac{n_{k}}{n_{k+1}}\right)=\infty
$$
\end{lem}

\begin{lem}\label{2.12} Suppose $n$ and $m$ are positive integers satifying $n<m$. It follows that

$$
\frac{2}{\pi}\left(1-\frac{n}{m}\right) \leq \sup _{\gamma}\left|\frac{1}{n} \sum\limits_{k=1}^{n} \gamma^{k}-\frac{1}{m} \sum\limits_{k=1}^{m} \gamma^{k}\right|
$$
\end{lem}

\begin{thm}\label{2.13} Suppose $\left(n_{k}\right)$ is any sequence of positive integers increasing to infinity. Suppose that $X$ is some probability space and that $\tau$ is an ergodic measure preserving transformation acting on $X$ and that $M_{n}$ is the usual ergodic averaging operator acting on $L^2(X)$ sending $f$ to $\frac{1}{n} \sum\limits_{k=1}^{n} f\left(\tau^{k} x\right)$, then

$$
\sum\limits_{k=1}^{\infty}\left\|M_{n_{k+1}}-M_{n_{k}}\right\|_{2}=\infty
$$
\end{thm}
\begin{proof}
 In what follows, ''m.p'' will stand for the phrase ''measure preserving''. Using the Conze Principle yields the fact that

$$
\begin{aligned}
& \sum\limits_{k=1}^{\infty}\left\|M_{n_{k+1}}^{\tau}-M_{n_{k}}^{\tau}\right\|_{2}=\sum\limits_{k=1}^{\infty} \sup _{\sigma \text { m.p. }}\left\|M_{n_{k+1}}^{\sigma}-M_{n_{k}}^{\sigma}\right\|_{2}
\end{aligned}
$$

$$
\begin{aligned}
& \ge\sum\limits_{k=1}^{\infty} \sup _{\sigma \text { m.p. } \sup _{\|\|_{2}=1}}\left(\int_{\mathbb{T}}\left|M_{n_{k+1}}^{\sigma} f(x)-M_{n_{k}}^{\sigma} f(x)\right|^{2} d x\right)^{1 / 2}
\end{aligned}
$$

Now for each $k$, let $\gamma_{k}$ be the point on $\mathbb{T}$ where

$$
\left|\frac{1}{n_{k+1}} \sum\limits_{j=1}^{n_{k+1}} \gamma_{k}^{j}-\frac{1}{n_{k}} \sum\limits_{j=1}^{n_{k}} \gamma_{k}^{j}\right|=\sup _{\gamma \in \mathbb{T}}\left|\frac{1}{n_{k+1}} \sum\limits_{j=1}^{n_{k+1}} \gamma^{j}-\frac{1}{n_{k}} \sum\limits_{j=1}^{n_{k}} \gamma^{j}\right|
$$

Let $\sigma_{\gamma_{k}}: \alpha \rightarrow \gamma_{k} \alpha$ and let $f(\alpha)=\alpha$. Then

$$
\begin{aligned}
\sup _{\sigma \text { m.p. } \sup _{\|\|_{2}=1}} \int_{\mathbb{T}}\left|M_{n_{k+1}}^{\sigma} f(x)-M_{n_{k}}^{\sigma} f(x)\right|^{2} d x & \geq \int_{\mathbb{T}}\left|\frac{1}{n_{k+1}} \sum\limits_{j=1}^{n_{k+1}}{\gamma_{k}}^{j} \alpha-\frac{1}{n_{k}} \sum\limits_{j=1}^{n_{k}} \gamma_{k}{ }^{j} \alpha\right|^{2} d \alpha \\
& =\left|\frac{1}{n_{k+1}} \sum\limits_{j=1}^{n_{k+1}} \gamma_{k}{ }^{j}-\frac{1}{n_{k}} \sum\limits_{j=1}^{n_{k}}{\gamma_{k}}^{j}\right|^{2}
\end{aligned}
$$

Thus

$$
\begin{aligned}
& \sum\limits_{k=1}^{\infty} \sup _{\sigma \text { m.p. }\|f\|_{2}=1} \sup _{\mathbb{T}}\left(\int_{\mathbb{1}}\left|M_{n_{k+1}}^{\sigma} f(x)-M_{n_{k}}^{\sigma} f(x)\right|^{2} d x\right)^{1 / 2} \geq \sum\limits_{k=1}^{\infty}\left|\frac{1}{n_{k+1}} \sum\limits_{j=1}^{n_{k+1}}{\gamma_{k}}^{j}-\frac{1}{n_{k}} \sum\limits_{j=1}^{n_{k}}{\gamma_{k}}^{j}\right| \\
& =\sum\limits_{k=1}^{\infty} \sup _{\gamma \in \mathbb{T}}\left|\frac{1}{n_{k+1}} \sum\limits_{j=1}^{n_{k+1}} \gamma^{j}-\frac{1}{n_{k}} \sum\limits_{j=1}^{n_{k}} \gamma^{j}\right| \\
& \geq\left(\frac{2}{\pi}\right) \sum\limits_{k=1}^{\infty}\left(1-\frac{n_{k}}{n_{k+1}}\right)=\infty
\end{aligned}
$$
\end{proof}

There is also a different way of proving Theorem~\ref{2.13} which is done by observing the fact that

$$
\sum\limits_{k=1}^{\infty}\left\|M_{n_{k+1}} f-M_{n_{k}} f\right\|_{2} \leq \sum\limits_{k=1}^{\infty}\left\|M_{n_{k+1}}-M_{n_{k}}\right\|_{2}\|f\|_{2}
$$

together with the following theorem.

\begin{thm}\label{2.14} Suppose $\left(n_{k}\right)$ is any sequence of integers increasing to $\infty$, and $\tau$ is an ergodic transformation. Then there is a function $f$ in $L^2$ such that

$$
\sum\limits_{k=1}^{\infty}\left\|M_{n_{k+1}} f-M_{n_{k}} f\right\|_{2}=\infty
$$
\end{thm}
\begin{proof}
Suppose not. Then by the Uniform Boundedness Principle, there is a constant $C$ such that

$$
\sum\limits_{k=1}^{\infty}\left\|M_{n_{k+1}} f-M_{n_{k}} f\right\|_{2} \leq C\|f\|_{2}
$$

By using Rokhlin's lemma, this can be transferred to $\mathbb{Z}$ and we get that for all $\varphi$ in $\ell_{2}(\mathbb{Z})$,

\begin{equation*}
\sum\limits_{k=1}^{\infty}\left\|M_{n_{k+1}} \varphi-M_{n_{k}} \varphi\right\|_{2} \leq C\|\varphi\|_{2} \tag{2.2}
\end{equation*}

Note that by repeatedly applying the triangle inequality, we can drop to a subsequence $\left(n_{k_{s}}\right)$ since

$$
\sum\limits_{s=1}^{\infty}\left\|M_{n_{k_{s+1}}} \varphi-M_{n_{k_{s}}} \varphi\right\|_{2} \leq \sum\limits_{k=1}^{\infty}\left\|M_{n_{k+1}} \varphi-M_{n_{k}} \varphi\right\|_{2} \leq C\|\varphi\|_{2}
$$

So without loss of generality, $2 n_{k} \leq n_{k+1}$. Thus by Lemma 2.12, it follows that for all $k$,

$$
\sup _{\gamma}\left|m_{n_{k+1}}(\gamma)-m_{n_{k}}(\gamma)\right| \geq \frac{2}{\pi}\left(1-\frac{n_{k}}{n_{k+1}}\right) \geq \frac{1}{\pi}
$$

For each $k$, let $\gamma_{k}$ be the point where the supremum is acheived. Note that the $M_{n}$ have asymptotically trivial transforms. This means that for all $\delta>0$, the sequence $\left(m_{n}(\gamma)\right)$ converges to zero uniformly on the set $S_{\delta}=\{\gamma \in \mathbb{T}:|\gamma-1| \geq \delta\}$. This can readily be seen from the fact that:

$$
\begin{aligned}
\left|m_{n}(\gamma)\right| & \leq \frac{1}{n}\left|\frac{\gamma-\gamma^{n+1}}{1-\gamma}\right| \\
& \leq \frac{2}{n \delta}
\end{aligned}
$$

\noindent which clearly goes to zero as $n$ goes to $\infty$.

Since $\left(M_{n}\right)$ has asymptotically trivial transforms, it follows that by taking an even sparser sequence of integers $\left(n_{k}\right)$ if necessary, it can be said that for some sequence $\left(n_{k}\right)$, there exists some $\alpha>0$ together with $M$ points $\gamma_{k}$ contained in pairwise disjoint $\operatorname{arcs} A_{k}$ each of length $\delta$ with

$$
\left|m_{n_{k+1}}(\gamma)-m_{n_{k}}(\gamma)\right| \geq \alpha>0
$$

\noindent for all $\gamma$ in the $k^{\text {th }}$ arc $A_{k}$. Note that the size of $\delta$ depends on $M$, but this will be seen to be of no consequence. Also by Lemma 2.12, it follows that $\alpha$ can be chosen independently of both $k$ and $M$.

Now, choose $\varphi$ in $\ell_{2}$ such that $\widehat{\varphi}=\sum\limits_{k=1}^{M} 1_{A_{k}}$. Then

$$
\begin{aligned}
C \sqrt{M} \sqrt{\delta}=C\left(\sum\limits_{k=1}^{M} \delta\right)^{1 / 2} & =C\|\widehat{\varphi}\|_{2}=C\|\varphi\|_{2} \\
& \geq \sum\limits_{k=1}^{\infty}\left\|M_{n_{k+1}} \varphi-M_{n_{k}} \varphi\right\|_{2} \quad \text { by Equation }(2.2) \\
& =\sum\limits_{k=1}^{\infty}\left(\int\left|m_{n_{k+1}}(\gamma)-m_{n_{k}}(\gamma)\right|^{2}|\widehat{\varphi}(\gamma)|^{2} d \gamma\right)^{1 / 2} \\
& \geq \sum\limits_{k=1}^{M}\left(\sum\limits_{j=1}^{M} \int_{A_{j}}\left|m_{n_{k+1}}(\gamma)-m_{n_{k}}(\gamma)\right|^{2} d \gamma\right)^{1 / 2} \\
& \geq \sum\limits_{k=1}^{M}\left(\int_{A_{k}}\left|m_{n_{k+1}}(\gamma)-m_{n_{k}}(\gamma)\right|^{2} d \gamma\right)^{1 / 2} \\
& \geq \sum\limits_{k=1}^{M}\left(\alpha^{2} \delta\right)^{1 / 2}=\alpha \sqrt{\delta} M
\end{aligned}
$$

Note that the $\delta$ 's cancel out and a contradiction is obtained since $M$ is arbitrarily large.
\end{proof}

There is also a second proof of Theorem~\ref{2.5} which uses the $D_{\epsilon}$ version of Theorem~\ref{2.14} which is stated and proved below:

\begin{thm}\label{2.15}
Suppose $\epsilon_{k}$ is any sequence of positive reals decreasing to $0$ . There is a function $f$ in $L^2$ such that

$$
\sum\limits_{k=1}^{\infty}\left\|D_{\epsilon_{k+1}} f-D_{\epsilon_{k}} f\right\|_{2}=\infty
$$
\end{thm}
\begin{proof}
Suppose not. Then by the Uniform Boundedness Principle, there is a constant $C$ such that for all $f$ in $L^2$,

\begin{equation*}
\sum\limits_{k=1}^{\infty}\left\|D_{\epsilon_{k+1}} f-D_{\epsilon_{k}} f\right\|_{2} \leq C\|f\|_{2} \tag{2.3}
\end{equation*}

Fix any $M$, and for $1 \leq k \leq M$, let

$$
g_{k}(x)=\frac{1_{\left(-\epsilon_{k}, \epsilon_{k}\right)}(x)}{2 \epsilon_{k}}-\frac{1_{\left(-\epsilon_{k+1}, \epsilon_{k+1}\right)}(x)}{2 \epsilon_{k+1}} .
$$

For all $\delta$, let $f_{\delta}(x)=\frac{1_{(-\delta, \delta)}(x)}{\sqrt{2 \delta}}$ so that $\left\|f_{\delta}\right\|_{2}=1$. Since $D_{\epsilon_{k}} f=g_{k} * f_{\delta}$, inequality (2.3) implies that for all $\delta$,

$$
\sum\limits_{k=1}^{M}\left(\int\left|g_{k} * f_{\delta}\right|^{2} d x\right)^{1 / 2} \leq C
$$

But for each $k, g_{k} \in L^2$. So $\left\|g_{k} * f_{\delta}-g_{k}\right\|_{2} \rightarrow 0$ as $\delta \rightarrow 0$. Thus

$$
\sum\limits_{k=1}^{M}\left(\int\left|g_{k}\right|^{2} d x\right)^{1 / 2} \leq C
$$

which is equivalent to saying that for all $M$,

$$
\frac{1}{2} \sum\limits_{k=1}^{M}\left(\int\left|\frac{1_{\left(-\epsilon_{k}, \epsilon_{k}\right)}(x)}{\epsilon_{k}}-\frac{1_{\left(-\epsilon_{k+1}, \epsilon_{k+1}\right)}(x)}{\epsilon_{k+1}}\right|^{2} d x\right)^{1 / 2} \leq C
$$

which, in turn is equivalent to saying that for all $M$

$$
\frac{1}{2} \sum\limits_{k=1}^{M}\left(2\left[\left(\frac{1}{\epsilon_{k+1}}-\frac{1}{\epsilon_{k}}\right)^{2} \epsilon_{k+1}+\frac{\epsilon_{k}-\epsilon_{k+1}}{\epsilon_{k}^{2}}\right]\right)^{1 / 2} \leq C .
$$

The left hand side of which simplies to

$$
\frac{1}{\sqrt{2}} \sum\limits_{k=1}^{M}\left(\frac{1}{\epsilon_{k+1}}-\frac{1}{\epsilon_{k}}\right)^{1 / 2}
$$

But for all $M$,

$$
\begin{aligned}
\sum\limits_{k=K_{o}}^{M} \sqrt{\left(\frac{1}{\epsilon_{k+1}}-\frac{1}{\epsilon_{k}}\right)} & \geq \sqrt{\sum\limits_{k=K_{o}}^{M}\left(\frac{1}{\epsilon_{k+1}}-\frac{1}{\epsilon_{k}}\right)} \\
& =\sqrt{\frac{1}{\epsilon_{M+1}}-\frac{1}{\epsilon_{1}}}
\end{aligned}
$$

Upon observing that $\lim _{M \rightarrow \infty} \sqrt{\frac{1}{\epsilon_{M+1}}-\frac{1}{\epsilon_{1}}}=\infty$, we see that we have a contradiction.
\end{proof}
Note that Theorem~\ref{2.15} provides a second proof of Theorem~\ref{2.5} if one makes observations analogous to those made just prior to the statement of Theorem~\ref{2.14}.

Theorem~\ref{2.15} and Theorem~\ref{2.14} also hold true for more general $p$.  For example, the following theorem is true:

\begin{thm}\label{2.16}  If $1<p<\infty$, then for all sequences $\left\{\epsilon_{k}\right\}$ converging down to zero, there exists an $f$ in $L^p$ such that

$$
\sum\limits_{k=1}^{\infty}\left\|D_{\epsilon_{k+1}} f-D_{\epsilon_{k}} f\right\|_{p}=\infty
$$
\end{thm}
\begin{proof}
Akcoglu, Jones, and Schwartz~\cite{AJS} have shown that if $\left\{\epsilon_{k}\right\}$ is a sequence satisfying $\frac{\epsilon_{k+1}}{\epsilon_{k}} \leq$ $\delta<1$, then there is an $f$ in $L_{\infty}$ such that $\sum\limits_{k=1}^{\infty}\left|D_{\epsilon_{k+1}} f-D_{\epsilon_{k}} f\right|=\infty$ a.e.

Choosing a subsequence $\left\{\epsilon_{k_{s}}\right\}$ of $\left\{\epsilon_{k}\right\}$ so that $\frac{\epsilon_{k_{s+1}}}{\epsilon_{k_{s}}} \leq \delta<1$, then

$$
\begin{aligned}
\sum\limits_{k=1}^{\infty}\left\|D_{\epsilon_{k+1}} f-D_{\epsilon_{k}} f\right\|_{p} & \geq \sum\limits_{s=1}^{\infty}\left\|D_{\epsilon_{k_{s+1}}} f-D_{\epsilon_{k_{s}}} f\right\|_{p} \\
& \geq \sum\limits_{s=1}^{\infty}\left\|D_{\epsilon_{k_{s+1}}} f-D_{\epsilon_{k_{s}}} f\right\|_{1}=\int\left(\sum\limits_{s=1}^{\infty}\left|D_{\epsilon_{k_{s+1}}} f-D_{\epsilon_{k_{s}}} f\right|\right) d m=\infty
\end{aligned}
$$
\end{proof}

The proof given above for Theorem~\ref{2.16} avoids any difficulties in computation due to taking $p \neq 2$ at the expense of invoking the nontrivial result in ~\cite{AJS}.

The following theorem in conjunction with Theorem~\ref{2.2} proves that for any sequence of positive integers $\left(n_{k}\right)$ tending to infinity, there exists a function $f$ in $L^1$ such that the sum $\sum\limits_{k=1}^{\infty}\left(M_{n_{k+1}} f-M_{n_{k}} f\right)$ is not weakly unconditionally convergent (and hence not absolutely convergent.)

\begin{thm}\label{2.17}
Let $\tau$ be any ergodic transformation of a non-atomic probability space $(X, m)$ then for any sequence of positive integers $\left\{n_{k}\right\}$ tending to infinity, there exists a function $f$ in $L^1$ such that

$$
\sup _{M \geq 1} \sup _{\left|c_{k}\right| \leq 1}\left\|\sum\limits_{k=1}^{M} c_{k}\left(M_{n_{k+1}} f-M_{n_{k}} f\right)\right\|_{1}=\infty
$$
\end{thm}
\begin{proof}
Suppose not. In other words, suppose that for all $f$ in $L^1$,

$$
\sup _{M \geq 1\left|c_{k}\right| \leq 1} \sup _{k=1}\left\|\sum\limits_{k=1}^{M} c_{k}\left(M_{n_{k+1}} f-M_{n_{k}} f\right)\right\|_{1}<\infty
$$

The Uniform Boundedness Principle implies that there is some uniform constant $C$ such that:

$$
\sup _{M \geq 1} \sup _{\left|c_{k}\right| \leq 1}\left\|\sum\limits_{k=1}^{M} c_{k}\left(M_{n_{k+1}} f-M_{n_{k}} f\right)\right\|_{1} \leq C\|f\|_{1}
$$

Now replace the coefficients $c_{k}$ with the Rademacher functions $r_{k}(t)$ and then integrate both sides with respect to $t$. This implies that:

$$
\int_{0}^{1} \int_{X}\left|\sum\limits_{k=1}^{M} r_{k}(t)\left(M_{n_{k+1}} f-M_{n_{k}} f\right)\right| d x d t \leq \int_{0}^{1} C d t\|f\|_{1}
$$

Which by Fubini's Theorem becomes:

$$
\int_{X} \int_{0}^{1}\left|\sum\limits_{k=1}^{M} r_{k}(t)\left(M_{n_{k+1}} f-M_{n_{k}} f\right)\right| d t d x \leq \int_{0}^{1} C d t\|f\|_{1}=C\|f\|_{1}
$$

We now use Theorem 8.4 on p. 213 of Volume I of Zygmund~\cite{Zygmund} which says that there exist positive constants $C_{1}, C_{2}$ such that for any $\left(b_{k}\right)$ in $\ell_{2}(\mathbb{Z})$,

\begin{equation*}
C_{1}\left(\int_{0}^{1}\left|\sum\limits_{k=1}^{\infty} r_{k}(t) b_{k}\right|^{p} d t\right)^{1 / p} \leq\left(\sum\limits_{k=1}^{\infty}{b_{k}}^{2}\right)^{1 / 2} \leq C_{2}\left(\int_{0}^{1}\left|\sum\limits_{k=1}^{\infty} r_{k}(t) b_{k}\right|^{p} d t\right)^{1 / p} \tag{2.4}
\end{equation*}

for all $p>0$. But this implies that

$$
\int_{X}\left(\sum\limits_{k=1}^{M}\left(M_{n_{k+1}} f-M_{n_{k}} f\right)^{2}\right)^{1 / 2} \leq C C_{2}\|f\|_{1}
$$

for all $M \geq 1$. Thus, since all the terms of the sum are positive, the Monotone Convergence Theorem implies

$$
\int_{X}\left(\sum\limits_{k=1}^{\infty}\left(M_{n_{k+1}} f-M_{n_{k}} f\right)^{2}\right)^{1 / 2} \leq C C_{2}\|f\|_{1}
$$

which in turn implies that

$$
\int_{X} \sup _{k}\left|M_{n_{k+1}} f-M_{n_{k}} f\right| d x \leq C C_{2}\|f\|_{1} .
$$

Since $\tau$ is ergodic, a Rokhlin tower construction can be used to provide the following inequality:

$$
\sum\limits_{m \in \mathbb{Z}} \sup _{k}\left|M_{n_{k+1}} \varphi(m)-M_{n_{k}} \varphi(m)\right| \leq C C_{2}\|\varphi\|_{\ell_{1}}
$$

for all $\varphi$ in $\ell_{1}$. We recall that

$$
M_{n_{k}} \varphi(m)=\frac{1}{n_{k}} \sum\limits_{j=1}^{n_{k}} \varphi(m+j)
$$

A contradiction will be arrived at by letting $\varphi$ be $\delta_{o}$. Note that

$$
M_{n_{k}} \delta_{o}(m)=\frac{1}{n_{k}} \text { when }-m=1, \ldots, n_{k} \text {. }
$$

Thus for each fixed $m>0$,

$$
\sup _{k}\left|M_{n_{k+1}} \delta_{o}(-m)-M_{n_{k}} \delta_{o}(-m)\right| \geq \frac{1}{n_{j+1}}
$$

where $j$ is chosen so that $n_{j}<m \leq n_{j+1}$. Putting all of this together yields:

$$
\begin{aligned}
\infty>C C_{2} 1 & =C C_{2}\left\|\delta_{o}\right\|_{\ell_{1}} \\
& \geq \sum\limits_{m \in \mathbb{Z}} \sup _{k}\left|M_{n_{k+1}} \delta_{o}(m)-M_{n_{k}} \delta_{o}(m)\right| \\
& \geq \sum\limits_{j=1}^{\infty} \sum\limits_{n_{j}<m \leq n_{j+1}} \sup _{k}\left|M_{n_{k+1}} \delta_{o}(-m)-M_{n_{k}} \delta_{o}(-m)\right| \\
& \geq \sum\limits_{j=1}^{\infty} \sum\limits_{n_{j}<m \leq n_{j+1}} \frac{1}{n_{j+1}} \\
& =\sum\limits_{j=1}^{\infty} \frac{n_{j+1}-n_{j}}{n_{j+1}}=\infty
\end{aligned}
$$

The contradiction is obtained.
\end{proof}

The $D_{\epsilon}$ version of Theorem~\ref{2.17} also holds true and is stated below. Theorem~\ref{2.18} can be proved using Rademacher functions together with an approximation to the identity.

\begin{thm}\label{2.18} Let $\left\{\epsilon_{k}\right\}$ be a decreasing sequence of positive reals convergence to zero. Let $X$ be either $\mathbb{R}$ or $\mathbb{T}$. Then there exists a function $f$ in $L^1(X)$ such that

$$
\sup _{M \geq 1} \sup _{\left|c_{k}\right| \leq 1}\left\|\sum\limits_{k=1}^{M} c_{k}\left(D_{\epsilon_{k+1}} f-D_{\epsilon_{k}} f\right)\right\|_{1}=\infty
$$
\end{thm}

\section{Ergodic Averages}\label{3}

Let us now consider the usual averages in ergodic theory, $M_{n} f=\frac{1}{n} \sum\limits_{k=1}^{n} f \circ \tau^{k}$, where $\tau$ is a measure-preserving transformation of a probability space $(X, m)$. Let $\left(n_{k}\right)$ be a non-decreasing sequence. We are interested in determining when the series of differences $\sum\limits_{k=1}^{\infty} M_{n_{k+1}} f-M_{n_{k}} f$ is unconditionally convergent for all $f \in L^2(X, m)$. This is motivated partly by the fact that for any such $n_{k}$, a martingale difference series of the form $\sum\limits_{k=1}^{\infty} E_{n_{k+1}} f-E_{n_{k}} f$, for the martingale $E_{n} f$ in $L^2$, will be unconditionally convergent simply because of the orthogonality of the terms. But in the ergodic theory context, while the averages $M_{n} f$ mimic many of the properties of the martingale $E_{n} f$, orthogonality is certainly not one of those properties. Instead, in the ergodic theory case, it will be certain aspects of the spectral values of $M_{n}$ which are important. As in the previous section, we let $m_{n}(\gamma)$ denote the Fourier transform of $M_{n}$, defined by $m_{n}(\gamma)=\frac{1}{n} \sum\limits_{k=1}^{n} \gamma^{k}$ for all $\gamma \in \mathbb{T}$, where $\mathbb{T}=\{\gamma \in \mathbb{C}:|\gamma|=1\}$. We then have this general result.

\begin{thm}\label{3.1} Let $\tau$ be any ergodic transformation of a non-atomic probability space $(X, m)$. Given a non-decreasing sequence $\left(n_{k}\right)$, the following are equivalent:

(a) $\quad \sum\limits_{k=1}^{\infty} M_{n_{k+1}} f-M_{n_{k}} f$ is unconditionally convergent for all $f \in L^2(X, m)$,

b) $\quad \sup\limits_{|\gamma|=1} \sum\limits_{k=1}^{\infty}\left|m_{n_{k+1}}(\gamma)-m_{n_{k}}(\gamma)\right|<\infty$.
\end{thm}
\begin{proof}
Suppose that (a) holds. Then applying Theorem~\ref{2.4}, we have some constant $C$ such that for all $M \geq 1,\left(c_{k}\right),\left|c_{k}\right| \leq 1$, and $f \in L^2(X, m)$,

$$
\left\|\sum\limits_{k=1}^{M} c_{k}\left(M_{n_{k+1}} f-M_{n_{k}} f\right)\right\|_{2} \leq C\|f\|_{2}
$$

This homogeneous bound is amenable to the Conze argument. That is, we have here used the averages $M_{n} f=M_{n}^{\tau} f=\frac{1}{n} \sum\limits_{k=1}^{n} f \circ \tau^{k}$. If we take an arbitrary measure-preserving transformation $\sigma$ and replace $f$ by $f \circ \sigma$, and then change integration by $\sigma$ too, we would then have the same inequality but with $M_{n} f$ replaced by $M_{n}^{\sigma \tau \sigma^{-1}} f=\frac{1}{n} \sum\limits_{k=1}^{n} f \circ\left(\sigma \tau \sigma^{-1}\right)^{k}$

throughout. That is, we would have with the same constant $C$, for all $M \geq 1,\left(c_{k}\right),\left|c_{k}\right| \leq 1$, and $f \in L^2(X, m)$,

$$
\left\|\sum\limits_{k=1}^{M} c_{k}\left(M_{n_{k+1}}^{\sigma \tau \sigma-1} f-M_{n_{k}}^{\sigma \tau \sigma^{-1}} f\right)\right\|_{2} \leq C\|f\|_{2}
$$

So, via the weak approximation of the general measure-preserving transformation $\tau_{0}$ by conjugates of the form $\sigma \tau \sigma^{-1}$, we have a constant $C$ such that for all measure-preserving transformations $\tau_{0}$, for all $M \geq 1,\left(c_{k}\right),\left|c_{k}\right| \leq 1$, and $f \in L^2(X, m)$,

$$
\left\|\sum\limits_{k=1}^{M} c_{k}\left(M_{n_{k+1}}^{\tau_{0}} f-M_{n_{k}}^{\tau_{0}} f\right)\right\|_{2} \leq C\|f\|_{2}
$$

Now let $\tau_{0}: \mathbb{T} \rightarrow \mathbb{T}$ by $\tau_{0}(\alpha)=\gamma \alpha,|\alpha|=1$, for some fixed $\gamma,|\gamma|=1$. Let $f(\alpha)=\alpha$. Then $M_{n}^{\tau_{0}} f(\alpha)=m_{n}(\gamma) \alpha$. So

$$
\sum\limits_{k=1}^{M} c_{k}\left(M_{n_{k+1}}^{\tau_{0}} f-M_{n_{k}}^{\tau_{0}} f\right)(\alpha)=\sum\limits_{k=1}^{M} c_{k}\left(m_{n_{k+1}}(\gamma)-m_{n_{k}}(\gamma)\right) \alpha
$$

Hence,

$$
\left\|\sum\limits_{k=1}^{M} c_{k}\left(M_{n_{k+1}}^{\tau_{0}} f-M_{n_{k}}^{\tau_{0}} f\right)\right\|_{2}=\left|\sum\limits_{k=1}^{M} c_{k}\left(m_{n_{k+1}}(\gamma)-m_{n_{k}}(\gamma)\right)\right|
$$

Thus, (a) gives a uniform bound for all $\gamma,|\gamma|=1$,

$$
\sup _{M \geq 1} \sup _{\left|c_{k}\right| \leq 1}\left|\sum\limits_{k=1}^{M} c_{k}\left(m_{n_{k+1}}(\gamma)-m_{n_{k}}(\gamma)\right)\right| \leq C
$$

But then, choosing $\left(c_{k}\right),\left|c_{k}\right|=1$, such that

$$
c_{k}\left(m_{n_{k+1}}(\gamma)-m_{n_{k}}(\gamma)\right)=\mid\left(m_{n_{k+1}}(\gamma)-m_{n_{k}}(\gamma) \mid\right.
$$

we would conclude that for all $\gamma,|\gamma|=1$,

$$
\sum\limits_{k=1}^{\infty}\left|m_{n_{k+1}}(\gamma)-m_{n_{k}}(\gamma)\right| \leq C
$$

Conversely, assume the bound in (b). Then consider the dynamical system $\tau_{0}: \mathbb{T} \rightarrow \mathbb{T}$, given by $\tau_{0}(\alpha)=\gamma \alpha$, where $\gamma$ is of $\infty$ order (so that $\tau_{0}$ is ergodic). We claim that if

$$
C=\sup _{|\gamma|=1} \sum\limits_{k=1}^{\infty}\left|m_{n_{k+1}}(\gamma)-m_{n_{k}}(\gamma)\right|
$$

then in the dynamical system $\left(\mathbb{T}, \tau_{0}\right)$, for all $M \geq 1$, and all $\left(c_{k}\right),\left|c_{k}\right| \leq 1$, we have for any $f \in L^2(T)$,

$$
\left\|\sum\limits_{k=1}^{M} c_{k}\left(M_{n_{k+1}}^{\tau_{0}} f-M_{n_{k}}^{\tau_{0}} f\right)\right\|_{2} \leq C\|f\|_{2}
$$

If so, via weak approximation as in the Conze method, we would have the same inequality for all $\tau$, which is the meaning of (a).

But indeed, here for fixed $\left(c_{k}\right)$ and $M$, denoting the Fourier transform of $F \in L^2(\mathbb{T})$, by $F T(F)$ or $\widehat{F}$,

$$
\left\|\sum\limits_{k=1}^{M} c_{k}\left(M_{n_{k+1}}^{\tau_{0}} f-M_{n_{k}}^{\tau_{0}} f\right)\right\|_{2}^{2}=\left\|F T\left(\sum\limits_{k=1}^{M} c_{k}\left(M_{n_{k+1}}^{\tau_{0}} f-M_{n_{k}}^{\tau_{0}} f\right)\right)\right\|_{l_{2}(\mathbb{Z})}^{2} .
$$

But here $\widehat{M_{n}^{\tau_{0}} f}(z)=m_{n}\left(\gamma^{z}\right) \widehat{f}(z)$. So we have

$$
\begin{aligned}
\left\|\sum\limits_{k=1}^{M} c_{k}\left(M_{n_{k+1}}^{\tau_{0}} f-M_{n_{k}}^{\tau_{0}} f\right)\right\|_{2}^{2} & =\sum\limits_{z=-\infty}^{\infty}\left|F T\left(\sum\limits_{k=1}^{M} c_{k}\left(M_{n_{k+1}}^{\tau_{0}} f-M_{n_{k}}^{\tau_{0}} f\right)\right)(z)\right|^{2} \\
& =\sum\limits_{z=-\infty}^{\infty}\left|\sum\limits_{k=1}^{M} c_{k}\left(m_{n_{k+1}}\left(\gamma^{z}\right)-m_{n_{k}}\left(\gamma^{z}\right)\right)\right|^{2}|\widehat{f}(z)|^{2} \\
& \leq \sum\limits_{z=-\infty}^{\infty}\left(\sum\limits_{k=1}^{M}\left|m_{n_{k+1}}\left(\gamma^{z}\right)-m_{n_{k}}\left(\gamma^{z}\right)\right|\right)^{2}|\widehat{f}(z)|^{2} \\
& \leq C^{2} \sum\limits_{z=-\infty}^{\infty}|\widehat{f}(z)|^{2} \\
& =C^{2}\|f\|_{2}^{2} .
\end{aligned}
$$
\end{proof}

We see then that the uniform boundedness of the series of the absolute values of the spectral differences is the essential criterion needed to determine unconditional convergence in $L^2$ of the associated differences in an arbitrary dynamical system. Here are two examples of the use of this criterion.
\medskip

\begin{exmp}~\label{3.2} Let $n_{k}=k$ for all values of $k$. We claim that Theorem~\ref{3.1}, (b) fails to hold. Indeed, in this case

$$
\left|m_{k+1}(\gamma)-m_{k}(\gamma)\right|=\frac{1}{k+1}\left|\gamma^{k+1}-m_{k}(\gamma)\right|
$$

Fix $\gamma \neq 1$. Then for large $k,\left|m_{k}(\gamma)\right| \leq \frac{1}{2}$. Hence,

$$
\frac{1}{k+1}\left|\gamma^{k+1}-m_{k}(\gamma)\right| \geq \frac{1}{2(k+1)}
$$

for sufficiently large values of $k$. Thus, for all $\gamma \neq 1, \sum\limits_{k=1}^{\infty}\left|m_{k+1}(\gamma)-m_{k}(\gamma)\right|=\infty$
\qed
\end{exmp}
\medskip

Thus, $\sum\limits_{k=1}^{\infty} M_{k+1} f-M_{k} f$ is never unconditionally convergent, for all $f \in L^2$, if $\tau$ is ergodic. Actually, since the argument in Theorem~\ref{2.4} uses the Baire Category Theorem, it is actually the case that for $\tau$ ergodic, this series is not unconditional convergent for a second category set of functions $f$. 

\begin{rem}
A similar conclusion can be made in the case that the sequence $\left(n_{k}\right)$ is a polynomial function of $k$, like $n_{k}=k^{N}$ for all $k$. However, it is not generally as clear whether there exists $\gamma$ for which the series in Theorem~\ref{3.1} (b) is infinite, but it is the case that there is no bound for the series which is independent of $\gamma$.  See the follow Section~\ref{PolyEGs} for the details.
\qed
\end{rem}
\medskip

\begin{exmp}~\label{3.3} Let $n_{k}=2^{k}$ for all values of $k$. We claim that Theorem~\ref{3.1} (b) holds. To see this we use the two estimates

(a) $\left|m_{n}(\gamma)\right| \leq \frac{2}{n|\gamma-1|}$,

(b) $\quad\left|m_{n}(\gamma)-1\right| \leq n|\gamma-1|$.

So

$$
\sum\limits_{k=1}^{\infty}\left|m_{n_{k+1}}(\gamma)-m_{n_{k}}(\gamma)\right| \leq \sum\limits_{k=1}^{L}\left|m_{n_{k+1}}(\gamma)-m_{n_{k}}(\gamma)\right|+\sum\limits_{k=L+1}^{\infty}\left|m_{n_{k+1}}(\gamma)-m_{n_{k}}(\gamma)\right|
$$

But

$$
\begin{aligned}
\sum\limits_{k=1}^{L}\left|m_{n_{k+1}}(\gamma)-m_{n_{k}}(\gamma)\right| & \leq \sum\limits_{k=1}^{L}\left|m_{n_{k+1}}(\gamma)-1\right|+\sum\limits_{k=1}^{L}\left|m_{n_{k}}(\gamma)-1\right| \\
& \leq 2 \sum\limits_{k=1}^{L} n_{k+1}|\gamma-1|
\end{aligned}
$$

Also,

$$
\sum\limits_{k=L+1}^{\infty}\left|m_{n_{k+1}}(\gamma)-m_{n_{k}}(\gamma)\right| \leq \frac{4}{|\gamma-1|} \sum\limits_{k=L+1}^{\infty} \frac{1}{n_{k}}
$$

Fix $\gamma \neq 1$, and take $L$ which is the last value of $k$ with $n_{k+1}|\gamma-1| \leq 1$. If there are no such $k$, then let $L=0$. Then

$$
\sum\limits_{k=1}^{L} n_{k+1}|\gamma-1|=|\gamma-1| \sum\limits_{k=1}^{L} 2^{k+1} \leq|\gamma-1| 2^{L+2}
$$

But $2^{L+1}|\gamma-1| \leq 1$. So $\sum\limits_{k=1}^{L} n_{k+1}|\gamma-1| \leq 2$. Hence,

$$
\sum\limits_{k=1}^{L}\left|m_{n_{k+1}}(\gamma)-m_{n_{k}}(\gamma)\right| \leq 4
$$

Also,

$$
\sum\limits_{k=L+1}^{\infty}\left|m_{n_{k+1}}(\gamma)-m_{n_{k}}(\gamma)\right| \leq \frac{4}{|\gamma-1|} \sum\limits_{k=L+1}^{\infty} \frac{1}{2^{k}}=\frac{4}{|\gamma-1|} \frac{1}{2^{L}}
$$

But $|\gamma-1| 2^{L+2}>1$. So $|\gamma-1| 2^{L}>\frac{1}{4}$ and

$$
\sum\limits_{k=L+1}^{\infty}\left|m_{n_{k+1}}(\gamma)-m_{n_{k}}(\gamma)\right| \leq 16
$$

Hence, $\sum\limits_{k=1}^{\infty}\left|m_{n_{k+1}}(\gamma)-m_{n_{k}}(\gamma)\right| \leq 20$.

Thus, for any dynamical system, $\sum\limits_{k=1}^{\infty} M_{2^{k+1}} f-M_{2^{k}} f$ is unconditionally convergent for all $f \in L^2$. Moreover, a very similar argument will show that as long as $\left(n_{k}\right)$ is lacunary, then $\sum\limits_{k=1}^{\infty}\left|m_{n_{k+1}}(\gamma)-m_{n_{k}}(\gamma)\right|$ is bounded for all $\gamma$ by a constant depending only on the lacunarity of the sequence $\left(n_{k}\right)$. So if $\left(n_{k}\right)$ is lacunary, then for any dynamical system, $\sum\limits_{k=1}^{\infty} M_{n_{k+1}} f-M_{n_{k}} f$ is unconditionally convergent for all $f \in L^2$.
\qed
\end{exmp}
\medskip

\subsection{More Examples with Some Details}\label{PolyEGs}

We consider more examples using some modifications of the the details that appear in Argiris~\cite{Argiris}.  In Argiris~\cite{Argiris} there are a number of results that are trying to get more definitively a characterization of the basic bound in Theorem~\ref{3.1}.

We need to examine the series

$$
\mathcal A = \sum_{k=1}^{\infty}\left |m_{n_{k+1}}(e^{i \theta})-m_{n_{k}}(e^{i \theta})\right  | 
$$
$$ =  \sum\limits_{k=1}^\infty \left |\frac 1{n_{k+1}}\left (\frac {e^{in_{k+1}\theta} - 1}{e^{i\theta}-1}\right ) -\frac 1{n_k}\left (\frac {e^{in_k \theta} -1 }{e^{i\theta} -1}\right )\right |. $$
We will be considering the bound for this series when $\sum\limits_{k=1}^\infty \frac 1{n_k} < \infty$.  In this case there is a uniform bound on 
\[\sup\limits_{0 < |\theta| < 1} \sum\limits_{k=1}^\infty \frac 1{n_k}
\left |\frac {e^{in_k\theta} - 1}{e^{i\theta} -1} - \frac {e^{in_k\theta} - 1}{\theta}\right | 
< \infty.\] 
Hence, when dealing with a uniform bound on $\mathcal A$, we may instead consider uniform bounds on 
\[\mathcal B = \sum\limits_{k=1}^\infty \left |\left (\frac {e^{in_{k+1}\theta} - 1}{n_{k+1}\theta}\right ) -\left (\frac {e^{in_k \theta} -1 }{n_k\theta}\right )\right |.\]
For $0 < |\theta|<1 / 2$, $\mathcal B
=\sum\limits_{k=1}^{\infty}\left|\int_{n_{k} \theta}^{n_{k+1} \theta} F^{\prime}(t) d t\right|$
where $F(t)=\frac{e^{i t}-1}{t}$. Let $F(t)=A(t)+i B(t)$. Then

\begin{align*}
F^{\prime}(t)&=A^{\prime}(t)+i B^{\prime}(t)\cr 
&= \frac{i t e^{i t}-\left(e^{i t}-1\right)}{t^{2}}\cr
&=\frac{1}{t^{2}}[i t(\cos t+i \sin t)-\cos t-i \sin t+1]\cr
&= \frac{1}{t^{2}}[i(t \cos t-\sin t)-(\cos t+t \sin t)+1]
\end{align*}

For large $t$, that is when $n_{k} \theta$ is large, we have

$$
\left|A^{\prime}(t)\right| \sim \frac{|\sin t|}{t}, \quad\left|B^{\prime}(t)\right| \sim \frac{|\cos t|}{t} .
$$

So, by the Mean Value Theorem, for some $t_1,t_2$ with 
$n_{k} \theta<t_{1}, t_{2}<n_{k+1} \theta$, for some constant $C$,

\begin{align*}
\mathcal B &= \sum_{k=1}^{\infty}\left|\int_{n_{k} \theta}^{n_{k+1} \theta} A^{\prime}(t) d t+i \int_{n_{k} \theta}^{n_{k+1} \theta} B^{\prime}(t) d t\right|  \cr
& \geq C\sum\limits_{k=1}^\infty \left|\left(n_{k+1}-n_{k}\right) \theta\right| \max \left(\frac{\left|\sin t_{1}\right|}{t_{1}}, \frac{\left|\cos t_{2}\right|}{t_{2}}\right).
\end{align*}

\noindent  Suppose $\left|t_{1}-t_{2}\right|$ is small, that is when $\left(n_{k+1}-n_{k}\right) \theta$ is small (for example, with $\left(n_{k+1}-n_{k}\right) \theta<1 / 10$).  Because at least one of 
$\cos t, \sin t$ is bounded away from $0$, again for some constant $C$, 
$$
\left|\left(n_{k+1}-n_{k}\right) \theta\right| \max \left(\frac{\left|\sin t_{1}\right|}{t_{1}}, \frac{\left|\cos t_{2}\right|}{t_{2}}\right) \ge
C \frac{n_{k+1}-n_{k}}{n_{k+1}} .
$$

\noindent Since $\theta$ can be arbitrarily small, for the sake of simplicity we will write the last condition as $\left(n_{k+1}-n_{k}\right) \theta<1$.

So we have proved

\begin{prop}\label{ArgirisPoly} When $\sum\limits_{k=1}^\infty \frac 1{n_k}$, we do not have a uniform bound on $\mathcal A$, if

$$
\sup _{1 / 2>\theta>0} \sum_{\substack{n_{k} \theta>100,\left(n_{k+1}-n_{k}\right) \theta<1}} \frac{n_{k+1}-n_{k}}{n_{k+1}}=\infty .
$$
\end{prop}

\begin{exmp}\label{PolyBad} Suppose $n_{k}=k^{p}$ for some $p>1$. Then $n_{k+1}-n_{k} \approx k^{p-1}$. So essentially the supremum in Proposition~\ref{ArgirisPoly} is over the set $\left\{k: k^{p-1} \theta<1, k^{p} \theta>100\right\} \approx\left\{k: k<\left(\frac{1}{\theta}\right)^{\frac{1}{p-1}}, k>\left(\frac{100}{\theta}\right)^{\frac{1}{p}}\right\}$.  But then the sum in this supremum is as large as 

$$
 \sum_{\left(\frac{100}{\theta}\right)^{1/p}<k<\left(\frac{1}{\theta}\right)^{\frac{1}{p-1}}} \frac{1}{k} \ge C \log \frac{\left(\frac{1}{\theta}\right)^{\frac{1}{p-1}}}{\left(\frac{100}{\theta}\right)^{\frac{1}{p}}} \ge C \frac{p}{p-1} \log \frac{1}{\theta} \rightarrow \infty
\quad\text {as}\quad \theta \to 0.$$

\noindent This example shows that for index sequences of the form $n_{k}=k^{p}, p>1$ we do not have unconditional convergence for the ergodic differences.
\end{exmp}

\begin{exmp}\label{TermsinLac} Suppose that the subsequence $\left(n_{k}\right)$ has $\phi(k)$ equally spaced terms between $2^{k}$ and $2^{k+1}$, that is, when $n_{j} \in\left[2^{k}, 2^{k+1}\right)$, we have $n_{j+1}-n_{j} \approx 2^{k} / \phi(k)$. Assume $\phi(k) \to \infty$ as $k\to \infty$.  We need to sum over the set $\left\{\left(n_{j+1}-n_{j}\right) \theta<1, n_{j} \theta>100\right\}$. Fix $k_{0}$ and let $\frac{1}{\theta}=2^{k_{0}}$. We are looking for $n_{j}>2^{k_{0}}$ such that $\left(n_{j+1}-n_{j}\right) \theta<1$, which is the same as $n_{j+1}-n_{j}<2^{k_{0}}$. For a fixed constant $C$, we can choose $k_{1}$ be such that $2^{k_{1}-k_{0}} \ge C \phi\left(k_{1}\right)$. Then our conditions are satisfied for all $n_{j}<2^{k_{1}}$ and

$$
\sum_{n_{j} \in[2^{k_{0}}, 2^{k_{1}})} \frac{n_{j+1}-n_{j}}{n_{j+1}} \ge C \sum_{k=k_{0}}^{k_{1}} \phi(k) \frac{2^{k} / \phi(k)}{2^{k}} \ge C (k_{1}-k_{0}) \ge C \log \left(\phi\left(k_{1}\right)\right) \rightarrow \infty.
$$

\noindent It follows from Theorem~\ref{ArgirisPoly} that when we have an index sequence $\left(n_{k}\right)$ such that it contains $\phi(k)$ equally spaced terms between $2^{k}$ and $2^{k+1}$, with $\phi(k) \rightarrow \infty$, we do not have unconditional convergence for the ergodic differences.
\end{exmp}

\section{Sequences of Probability Measures}\label{4}

It is clear from the argument in Theorem~\ref{3.1}, that the averages $M_{n}$ can be replaced by a more general sequence of averages. In particular, suppose that $\mu$ is a probability measure on $\mathbb{Z}$ and let $\mu f$ denote the usual average, given a fixed dynamical system $(X, m, \tau)$. That

is, $\mu f(x)=\sum\limits_{s=-\infty}^{\infty} \mu(s) f\left(\tau^{s}(x)\right)$. We also let $\widehat{\mu}$ denote the usual Fourier transform of $\mu$. Then just as the above is proved, one can show that

\begin{thm}\label{4.1} Let $\tau$ be any ergodic transformation of a non-atomic probability space $(X, m)$. Given a sequence of probability measures $\left(\mu_{k}\right)$, the following are equivalent:

(a) $\quad \sum\limits_{k=1}^{\infty} \mu_{k+1} f-\mu_{k} f$ is unconditionally convergent for all $f$ in $L^2(X, m)$,

b) $\quad \sup _{|\gamma|=1} \sum\limits_{k=1}^{\infty}\left|\widehat{\mu}_{k+1}(\gamma)-\widehat{\mu}_{k}(\gamma)\right|<\infty$.
\end{thm}

\noindent {\bf Remark}:  Even more generally, the differences $\mu_{k+1} f-\mu_{k} f=\left(\mu_{k+1}-\mu_{k}\right) f$ in Theorem~\ref{4.1} can be replaced by a fixed sequence of terms $\nu_{k} f$ where $\nu_{k}$ are finite signed measures on $\mathbb{Z}$.
\qed
\medskip

We are now about to present several interesting facts related to Theorem~\ref{4.1}.  But before we do, we should recall that if $\frac{|1-\hat{\mu}(\gamma)|}{1-|\hat{\mu}(\gamma)|} \leq C$, then the spectrum of $\mu$ is contained in some Stoltz region.

Furthermore, if $\frac{|\hat{\mu}(\gamma)-1|^{2}}{1-|\hat{\mu}(\gamma)|^{2}}$ is bounded by some number $M$, then it follows by couple lines of algebra that the spectrum of $\mu$ is contained in some circle $C_{\rho}$ with center at $(\rho, 0)$ and radius $1-\rho$ where $\rho=\frac{1}{M+1}$. It was first observed in the paper by Jones, Ostrovskii and Rosenblatt~\cite{JOR}  that ''strictly aperiodic'' measures all have a spectrum which is contained in some such horocycle.

\begin{lem}\label{4.2} If $\mu$ is a probability measure on $\mathbb{Z}$, and the spectrum of $\mu$ is contained in some Stolz region, then for all $\left(n_{k}\right)$ satisfying $n_{k} \leq n_{k+1}$, we have some constant $C$ such that

$$
\sup _{|\gamma|=1} \sum\limits_{k=1}^{\infty}\left|\widehat{\mu}^{n_{k+1}}(\gamma)-\widehat{\mu}^{n_{k}}(\gamma)\right| \leq C
$$
\end{lem}

The following proof is almost identical to the proof of Theorem 1.5 given in Jones, Ostrovskii and Rosenblatt~\cite{JOR} where it is proved that for some constant $C$,

$$
\sup _{|\gamma|=1} \sum\limits_{k=1}^{\infty}\left|\widehat{\mu}^{n_{k+1}}(\gamma)-\widehat{\mu}^{n_{k}}(\gamma)\right|^{2} \leq C
$$

\noindent This proof is repeated here for convenience:
\begin{proof}[Proof of Lemma~\ref{4.2}] 
We observe that

$$
\begin{aligned}
\sum\limits\left|z^{n_{k+1}}-z^{n_{k}}\right| & =\sum\limits|z|^{n_{k}}\left|z^{n_{k+1}-n_{k}}-1\right| \\
& =\sum\limits|z|^{n_{k}}|1-z|\left|1+z+z^{2}+\ldots+z^{n_{k+1}-n_{k}-1}\right| \\
& =|1-z| \sum\limits|z|^{n_{k}} \frac{1-|z|^{n_{k+1}-n_{k}}}{1-|z|} \\
& \leq C \sum\limits|z|^{n_{k}}-|z|^{n_{k+1}} \\
& =C|z|^{n_{1}}
\end{aligned}
$$

Substituting $z=\widehat{\mu}(\gamma)$ gives the desired result
\end{proof}

Now we need this additional lemma.

\begin{lem}~\label{4.3} If $\frac{|\hat{\mu}-1|^{2}}{|1-\hat{\mu}|} \leq C$ then $n_{k+1} \geq n_{k}{ }^{p}, p>1$ implies that

$$
\sup _{|\gamma|=1} \sum\limits_{k=1}^{\infty}\left|\widehat{\mu}^{n_{k+1}}(\gamma)-\widehat{\mu}^{n_{k}}(\gamma)\right|<\infty
$$
\end{lem}
\begin{proof}
Let $S_{L} =\sum\limits_{k=1}^{L}\left|\widehat{\mu}^{n_{k+1}}-\widehat{\mu}^{n_{k}}\right|$.  Then
$$
\begin{aligned}
S_{L}&\leq \sum\limits_{k=1}^{L}|\widehat{\mu}|^{n_{k}}\left|\widehat{\mu}^{n_{k+1}-n_{k}}-1\right| \\
& =\sum\limits_{k=1}^{L}|\widehat{\mu}|^{n_{k}}|\widehat{\mu}-1|\left|1+\hat{\mu}+\ldots+\hat{\mu}^{n_{k+1}-n_{k}-1}\right| 
 \leq \sum\limits_{k=1}^{L}|\widehat{\mu}|^{n_{k}}|\widehat{\mu}-1|\left(n_{k+1}-n_{k}\right) \\
& \leq|\widehat{\mu}-1| \sum\limits_{k=1}^{L}\left(n_{k+1}-n_{k}\right) 
 \leq|\widehat{\mu}-1| n_{L+1}
\end{aligned}
$$

Let $R_{L} =\sum\limits_{k=L+2}^{\infty}\left|\widehat{\mu}_{k+1}^{n}-\widehat{\mu}_{k}^{n}\right|$. 
$$
\begin{aligned}
R_{L} & \leq \sum\limits_{k=L+2}^{\infty}|\hat{\mu}|^{n_{k}}\left|\hat{\mu}^{n_{k+1}-n_{k}}-1\right| \\
& =\sum\limits_{k=L+2}^{\infty}|\hat{\mu}|^{n_{k}}|\hat{\mu}-1|\left|1+\hat{\mu}+\ldots+\hat{\mu}^{n_{k+1}-n_{k}-1}\right| 
 \leq \sum\limits_{k=L+2}^{\infty}|\hat{\mu}|^{n_{k}}|\hat{\mu}-1| \frac{1-|\hat{\mu}|^{n_{k+1}-n_{k}}}{1-|\mu|} \\
& =\frac{\hat{\mu}-1}{1-|\hat{\mu}|} \sum\limits_{k=L+2}^{\infty}|\hat{\mu}|^{n_{k}}-|\hat{\mu}|^{n_{k+1}}
\leq \frac{\hat{\mu}-1}{1-|\hat{\mu}|}|\hat{\mu}|^{n_{L+2}} 
 \leq 2 \frac{|\hat{\mu}|^{n_{L+2}}}{1-|\hat{\mu}|}
\end{aligned}
$$

We assume without loss of generality that $0<|\hat{\mu}-1|<1$, and for each $\gamma$, we choose $L(\gamma)$ such that

$$
n_{L+1} \leq \frac{1}{|\hat{\mu}-1|}<n_{L+2}
$$

Note that $S_{L} \leq 1$ and choose $d \geq 1$ such that $p^{d}>3$, then $n_{L+d+2} \geq n_{L+2} p^{p^{d}}$ thus

$$
\begin{aligned}
& \ln \frac{|\hat{\mu}|^{n_{L+d+2}}}{1-|\hat{\mu}|}=n_{L+d+2} \ln |\hat{\mu}|+\ln \frac{1}{1-|\hat{\mu}|} \\
& \leq n_{L+2}{ }^{p^{d}} \ln |\hat{\mu}|+\ln \frac{1}{1-|\hat{\mu}|} \\
& \leq \frac{\ln |\hat{\mu}|}{|\hat{\mu}-1|^{p^{d}}}+\ln \frac{1}{1-|\hat{\mu}|}
\end{aligned}
$$

Furthermore,

\[
\begin{array}{rlr}
|\hat{\mu}-1|^{p^{d}} & =|\hat{\mu}-1|^{p^{d}-1}|\hat{\mu}-1| \\
& \left.\leq|\hat{\mu}-1|^{2}|\hat{\mu}-1| \quad \text { (we're assuming }|\hat{\mu}-1|<1\right) \\
& \leq C(1-|\hat{\mu}|)|\hat{\mu}-1| \quad \text { (by hypothesis) } \tag{by hypothesis}
\end{array}
\]

Thus

$$
\frac{-(1-|\hat{\mu}|)}{|\hat{\mu}-1|^{p^{d}}} \leq \frac{-1}{C(1-|\hat{\mu}|)|\hat{\mu}-1|}
$$

and we have

\begin{equation*}
\frac{-(1-|\hat{\mu}|)}{|\hat{\mu}-1|^{p^{d}}} \leq \frac{-1}{C|\hat{\mu}-1|} \tag{4.2}
\end{equation*}

Also, by Taylor's expansion, we see that

\begin{equation*}
\ln |\hat{\mu}| \leq|\hat{\mu}|-1=-(1-|\hat{\mu}|) \tag{4.3}
\end{equation*}

Thus continuing from inequality (4.1), we have

$$
\begin{array}{rlr}
\ln \frac{|\hat{\mu}|^{n_{L}+d+2}}{1-|\hat{\mu}|} & \leq \frac{\ln |\hat{\mu}|}{|\hat{\mu}-1|^{p^{d}}}+\ln \frac{1}{1-|\hat{\mu}|} \\
& \leq \frac{-(1-|\hat{\mu}|)}{|\hat{\mu}-1|^{p^{d}}}+\ln \frac{1}{1-|\hat{\mu}|} \quad \text { by inequality (4.3) } \\
& \leq \frac{-1}{C|\hat{\mu}-1|}+\ln \frac{1}{1-|\hat{\mu}|} \quad \text { by inequality (4.2) } \\
& \leq \frac{-1}{C|\hat{\mu}-1|}+\ln \frac{C}{|\hat{\mu}-1|^{2}} \quad \text { since } \mu \text { is strictly aperiodic }
\end{array}
$$

But $\lim _{x \rightarrow 0}\left(\frac{-1}{c|x|}-2 \ln |x|\right)=-\infty$

Thus we've just shown that $\sup _{\gamma} R_{L+d}$ is finite and we recall that earlier in the proof we showed that

$$
\begin{aligned}
S_{L} & \leq|\hat{\mu}-1| n_{L+1} \\
& \leq \frac{|\hat{\mu}-1|}{|\hat{\mu}-1|}=1
\end{aligned}
$$

(here, the last inequality follows from how $\mathrm{L}$ was chosen)

Thus we see that

$$
\sup _{|\gamma|=1} \sum\limits_{k=1}^{\infty}\left|\widehat{\mu}^{n_{k+1}}(\gamma)-\widehat{\mu}^{n_{k}}(\gamma)\right| \leq S_{L}+2(d+1)+R_{L+d+2}
$$

is also finite and that $d$ is independent of either $L$ or $\gamma$
\end{proof}

A version of the following fact appears in the the paper of Jones, Ostrovskii, and Rosenblatt~\cite{JOR}

\begin{lem}~\label{4.4} There is a strictly aperiodic measure $\mu$ (namely the measure $\mu=\frac{\delta_{0}+\delta_{1}}{2}$ ) such that when $n_{k}=q^{k}$ with $q \geq 2$, it follows that

$$
\sup _{|\gamma|=1} \sum\limits_{k=1}^{\infty}\left|\hat{\mu}^{n_{k+1}}(\gamma)-\hat{\mu}^{n_{k}}(\gamma)\right|^{2}=\infty
$$
\end{lem}

\begin{lem}\label{4.5} If $n_{k}=q^{k}$, with $q \geq 2$ then there is a strictly aperiodic measure $\mu$ (namely the measure $\mu=\frac{\delta_{0}+\delta_{1}}{2}$ ) such that

$$
\sup _{|\gamma|=1} \sum\limits_{k=1}^{\infty}\left|\hat{\mu}^{n_{k+1}}(\gamma)-\hat{\mu}^{n_{k}}(\gamma)\right|=\infty
$$
\end{lem}
\begin{proof}The proof is trivial using Lemma~\ref{4.4} above since for any function $f$ defined on $\mathbb{Z},\|f\|_{2} \leq\|f\|_{1}$
\end{proof}

We recall that in Jones, Ostrovskii, and Rosenblatt~\cite{JOR}, it is shown that if $\mu$ is strictly aperiodic then for some constant $\mathrm{C}$,

$$
\sup _{|\gamma|=1} \sum\limits_{k=1}^{\infty}\left|\hat{\mu}^{n_{k+1}}(\gamma)-\hat{\mu}^{n_{k}}(\gamma)\right|^{2} \leq C
$$

\noindent when $n_{k}=k$. Thus it is interesting to note that when the exponent $2$ is decreased this boundedness is no longer guaranteed under just the condition that $\mu$ is strictly aperiodic:

\begin{lem}\label{4.6} There are strictly aperiodic measures $\mu$ which do not have the property that for some constant $C$,

$$
\sup _{|\gamma|=1} \sum\limits_{k=1}^{\infty}\left|\hat{\mu}^{k+1}(\gamma)-\hat{\mu}^{k}(\gamma)\right|<C
$$
\end{lem}
\begin{proof}  This is easy to see since

$$
\begin{aligned}
\sum\limits_{k=1}^{\infty}\left|\hat{\mu}^{k+1}-\hat{\mu}^{k}\right| & =\sum\limits_{k=1}^{\infty}|\hat{\mu}|^{k}|\hat{\mu}-1| \\
& =\frac{|\hat{\mu}-1||\hat{\mu}|}{1-|\hat{\mu}|}
\end{aligned}
$$

Thus we see that there is some constant $\mathrm{C}$ such that

$$
\sup _{|\gamma|=1} \sum\limits_{k=1}^{\infty}\left|\hat{\mu}^{k+1}(\gamma)-\hat{\mu}^{k}(\gamma)\right|<C
$$

if and only if $\frac{|\hat{\mu}-1|}{1-|\hat{\mu}|}$ is bounded. There are plenty of strictly aperiodic measures which do not satisfy: $\frac{|\hat{\mu}-1|}{1-|\hat{\mu}|}$ is bounded. The measure $\frac{\delta_{0}+\delta_{1}}{2}$ is such an example.
\end{proof}

It is shown in Theorem 10 of Rosenblatt~\cite{Rosenblatt} that if $\mu=\frac{\delta_{0}+\delta_{1}}{2}$, then for all $\epsilon>0$, there exists a set $E$, with $m(E)<\epsilon$ such that $\lim \sup \mu^{n} 1_{E}(x)=1$ a.e. In other words, it is shown that the sequence of measures $\left(\mu^{n}\right)_{n=1}^{\infty}$ is strongly sweeping out.

Furthermore it is shown in Corollary 11 of Bellow, Jones and Rosenblatt~\cite{BJR} that for all $f\in L^1$, $\lim\limits_{m\to \infty} \mu^{2^{2^m}}f(x)$ exists for a.e. $x$.

We note that we would have another proof of this fact if we could just prove that when $n_{k}=2^{2^{k}}$ and $\mu=\frac{\delta_{0}+\delta_{1}}{2}, \sup _{|\gamma|=1} \sum\limits_{k=1}^{\infty} \ln (k+1)\left|\hat{\mu}^{n_{k+1}}(\gamma)-\hat{\mu}^{n_{k}}(\gamma)\right|<\infty$. For it would then follow by a weighted version of Theorem~\ref{4.1} that $\sum\limits_{k=1}^{\infty} \ln (k+1)\left(\mu^{n_{k+1}} f-\mu^{n_{k}} f\right)$ is unconditionally convergent in $L^2$ for all $f \in L^2$. But since $\|f\|_{1} \leq\|f\|_{2}$, we see that $\sum\limits_{k=1}^{\infty} \ln (k+1)\left(\mu^{n_{k+1}} f-\mu^{n_{k}} f\right)$ is unconditionally convergent in $L^1$ for all $f \in L^2$. But by Theorem 2.5 in Bennett~\cite{Bennett}, it follows that $\sum\limits_{k=1}^{\infty}\left(\mu^{n_{k+1}} f(x)-\mu^{n_{k}} f(x)\right)$ converges for a.e. $x$ for all $f \in L^2$ Thus this would imply the desired: $\mu^{n_{k}} f(x)$ converges for a.e. $x$ whenever $f \in L^2$. The following lemma however rules out the possibility of using this line of reasoning to prove such an a.e. convergence result.

\begin{lem}\label{4.7} If $\mu=\frac{\delta_{0}+\delta_{1}}{2}, n_{k}$ is any sequence of positive integers satisfying

$$
2 n_{k}{ }^{p} \geq n_{k+1} \geq n_{k}^{p}, p>1
$$

and $\left\{w_{k}\right\}$ is any sequence of reals increasing to $\infty$ then it follows that

$$
\sup _{|\gamma|=1} \sum\limits_{k=1}^{\infty} w_{k}\left|\widehat{\mu}^{n_{k+1}}(\gamma)-\widehat{\mu}^{n_{k}}(\gamma)\right|=\infty
$$
\end{lem}

We note that the sequence $n_{k}=2^{2^{k}}$ satisfies the conditions of Lemma 4.7.

\begin{proof}[Proof of Lemma~\ref{4.7}] 
It suffices to show that if

\begin{equation*}
2 n_{k}{ }^{p} \geq n_{k+1} \geq n_{k}{ }^{p}, p>1 \tag{4.4}
\end{equation*}

then there is a sequence of positive real numbers $\left\{\gamma_{k}\right\}$ with $\gamma_{k} \rightarrow 0$ such that for some $\delta>0$, eventually

$$
\left|\widehat{\mu}_{n_{k+1}}\left(\gamma_{k}\right)-\widehat{\mu}_{n_{k}}\left(\gamma_{k}\right)\right| \geq \delta>0
$$

will hold true (when $k$ is big enough).

We observe that $\widehat{\mu}\left(e^{i t}\right)=e^{i t / 2} \cos (t / 2)$ thus

$$
\begin{aligned}
\left|\widehat{\mu}_{n_{k+1}}(\gamma)-\widehat{\mu}_{n_{k}}(\gamma)\right| & =|\hat{\mu}(\gamma)|^{n_{k}}\left|\hat{\mu}^{n_{k+1}-n_{k}}(\gamma)-1\right| \\
& \geq(\cos t / 2)^{n_{k}}\left(\sin \left(\frac{\left(n_{k+1}-n_{k}\right) t}{2}\right)\right)(\cos t / 2)^{n_{k+1}-n_{k}} \\
& =(\cos t / 2)^{n_{k+1}}\left(\sin \left(\frac{\left(n_{k+1}-n_{k}\right) t}{2}\right)\right)
\end{aligned}
$$

The above inequality follows from the fact that $|z| \geq \operatorname{Im}(z)$ for any complex number $\mathrm{z}$.

Next, choose $t_{k}$ such that

\begin{equation*}
\frac{\pi}{2 n_{k}{ }^{p}}<\frac{\pi}{n_{k+1}}<t_{k}<\frac{\pi}{n_{k+1}-n_{k}}<\frac{\pi}{n_{k}{ }^{p}-n_{k}} \tag{4.5}
\end{equation*}

We observe that

$$
\left(t_{k}<\frac{\pi}{n_{k+1}-n_{k}}\right) \Rightarrow\left(\frac{\left(n_{k+1}-n_{k}\right) t_{k}}{2}<\frac{\pi}{2}\right) \Rightarrow\left(\sin \frac{\left(n_{k+1}-n_{k}\right) t_{k}}{2} \geq \frac{2}{\pi} \frac{\left(n_{k+1}-n_{k}\right) t_{k}}{2}\right)
$$

Thus we see that

$$
\begin{aligned}
\left|\widehat{\mu}^{n_{k+1}}\left(e^{i t_{k}}\right)-\widehat{\mu}^{n_{k}}\left(e^{i t_{k}}\right)\right| & \geq\left(\cos t_{k} / 2\right)^{n_{k+1}}\left(\frac{\left(n_{k+1}-n_{k}\right) t_{k}}{\pi}\right) \\
& \geq\left(1-\frac{t_{k}^{2}}{8}\right)^{n_{k+1}}\left(\frac{\left(n_{k+1}-n_{k}\right) t_{k}}{\pi}\right)
\end{aligned}
$$

$$
\begin{aligned}
& \geq\left(1-\frac{n_{k+1} t_{k}^{2}}{8}\right)\left(\frac{\left(n_{k+1}-n_{k}\right) t_{k}}{\pi}\right) \\
& \geq\left(1-\frac{n_{k+1} t_{k}^{2}}{8}\right)\left(\frac{\left(n_{k+1}-n_{k}\right)}{n_{k+1}}\right)(\text { by (4.5)) } \\
& \geq\left(1-\frac{n_{k+1} \pi^{2}}{\left(n_{k+1}-n_{k}\right)^{2} 8}\right)\left(\frac{\left(n_{k+1}-n_{k}\right)}{n_{k+1}}\right)(\text { by equation (4.5) again) } \\
& =\left(\frac{\left(n_{k+1}-n_{k}\right)^{2}-n_{k+1} \pi^{2} / 8}{\left(n_{k+1}-n_{k}\right)^{2}}\right)\left(\frac{n_{k+1}-n_{k}}{n_{k+1}}\right) \\
& =\left(\frac{\left(n_{k+1}-n_{k}\right)^{2}-n_{k+1} \pi^{2} / 8}{\left(n_{k+1}-n_{k}\right) n_{k+1}}\right) \\
& \geq \frac{n_{k+1}{ }^{2}-2 n_{k} n_{k+1}+n_{k}{ }^{2}-2\left(n_{k}\right)^{p} \pi^{2} / 8}{\left(n_{k+1}-n_{k}\right) n_{k+1}} \\
& \geq \frac{n_{k}{ }^{2 p}-4 n_{k} n_{k}{ }^{p}+n_{k}{ }^{2}-2 n_{k} \pi^{2} / 8}{\left(2 n_{k}{ }^{p}-n_{k}\right) 2 n_{k}{ }^{p}} \\
& \geq \frac{n_{k}{ }^{2 p}-4 n_{k}{ }^{p+1}+n_{k}{ }^{2}-4 n_{k}{ }^{p}}{4 n_{k}{ }^{2 p}-2 n_{k}{ }^{p+1}} \\
& \longrightarrow \frac{1}{4}>0
\end{aligned}
$$

The last  $\geq$ inequality follows from the fact that $\pi^{2} / 8 \leq 16 / 8=2$ and the previous two inequalities follow from algebra and also from equation (4.4). The desired result follows by choosing $\delta$ to be any positive number less than $1 / 4$.
\end{proof}

\section{Unconditional Convergence for Differences of Derivatives}\label{5}

Let

$$
\varphi_{\epsilon}\left(e^{i \theta}\right)= \begin{cases}1 / 2 \epsilon & \text { when }-\epsilon \leq \theta \leq+\epsilon \\ 0 & \text { otherwise }\end{cases}
$$

then we define

$$
\begin{aligned}
\left(\varphi_{\epsilon} * f\right)\left(e^{i x}\right) & =\int \varphi_{\epsilon}\left(e^{i y}\right) f\left(e^{i(x-y)}\right) d y \\
& =\frac{1}{2 \epsilon} \int_{-\epsilon}^{+\epsilon} f\left(e^{i(x-y)}\right) d y
\end{aligned}
$$

and we see that

$$
\begin{aligned}
\widehat{\varphi_{\epsilon}}(\ell) & =\int e^{-i \ell \dot{x}} \varphi_{\epsilon}\left(e^{i x}\right) d x \\
& =\frac{\sin \ell \epsilon}{\ell \epsilon}
\end{aligned}
$$

Upon defining $D_{\epsilon} f=\varphi_{\epsilon} * f$, we can now state the following theorem:

\begin{thm}\label{5.1}
If $\epsilon_{k}$ is any sequence of positive reals tending to zero, then following are equivalent:

(a) $\quad \sum\limits_{k=1}^{\infty}\left(D_{\epsilon_{k+1}} f-D_{\epsilon_{k}} f\right)$ is unconditionally convergent in $L^2$ for all $f \in L^2(\mathbb{T})$

(b) There is a finite constant $C$ such that $\sup _{\{\ell \in \mathbb{Z}\}} \sum\limits_{k=1}^{\infty}\left|\widehat{\varphi}_{\epsilon_{k+1}}(\ell)-\widehat{\varphi}_{\epsilon_{k}}(\ell)\right| \leq C$
\end{thm}
\begin{proof}  Suppose that $\sum\limits_{k=1}^{\infty}\left(D_{\epsilon_{k+1}} f-D_{\epsilon_{k}} f\right)$ is unconditionally convergent in $L^2$ for all $f$ in $L^2$. Then by Theorem~\ref{2.4}, there is some finite number $C$ such that:

$$
\sup _{M} \sup _{\left|c_{k}\right| \leq 1}\left\|\sum\limits_{k=1}^{M} c_{k}\left(\varphi_{\epsilon_{k+1}} * f-\varphi_{\epsilon_{k}} * f\right)\right\|_{2} \leq C\|f\|_{2} \text { for all } f \text { in } L^2.
$$

Thus

$$
\sup _{M} \sup _{\left|c_{k}\right| \leq 1}\left\{\sum\limits_{\ell \in \mathbb{Z}}\left|\sum\limits_{k=1}^{M} c_{k}\left(\widehat{\varphi}_{\epsilon_{k+1}}(\ell) \widehat{f}(\ell)-\widehat{\varphi}_{\epsilon_{k}}(\ell) \widehat{f}(\ell)\right)\right|^{2}\right\}^{1 / 2} \leq C\left\{\sum\limits_{\ell \in \mathbb{Z}}|\hat{f}(\ell)|^{2}\right\}^{1 / 2}
$$

Fix $\ell_{o}$ and choose $f$ such that

$$
\hat{f}(\ell)= \begin{cases}1, & \text { for } \ell=\ell_{o} \\ 0, & \text { for } \ell \neq \ell_{o}\end{cases}
$$

Then

$$
\sup _{M} \sup _{\left|c_{k}\right| \leq 1}\left|\sum\limits_{k=1}^{M} c_{k}\left(\widehat{\varphi}_{\epsilon_{k+1}}\left(\ell_{o}\right)-\widehat{\varphi}_{\epsilon_{k}}\left(\ell_{o}\right)\right)\right|^{2} \leq C^{2}
$$

Thus we see that if we choose $c_{k}$ equal to the sign of $\left(\widehat{\varphi}_{\epsilon_{k+1}}\left(\ell_{o}\right)-\widehat{\varphi}_{\epsilon_{k}}\left(\ell_{o}\right)\right)$ then letting $M \rightarrow \infty$ gives us

$$
\sum\limits_{k=1}^{\infty}\left|\widehat{\varphi}_{\epsilon_{k+1}}\left(\ell_{o}\right)-\widehat{\varphi}_{\epsilon_{k}}\left(\ell_{o}\right)\right| \leq C
$$

Since the value of $\mathrm{C}$ is independent of how we choose $\ell_{o}$ we see that we are done with this direction of the proof.

Suppose now that

$$
\sup _{\ell \text { in } \mathbb{Z}} \sum\limits_{k=1}^{\infty}\left|\widehat{\varphi}_{\epsilon_{k+1}}(\ell)-\widehat{\varphi}_{\epsilon_{k}}(\ell)\right| \leq C
$$

then

$$
\begin{aligned}
& \sup _{M} \sup _{\left|c_{k}\right| \leq 1} \mid \sum\limits_{k=1}^{M} c_{k}\left(\varphi_{\epsilon_{k+1}} * f-\varphi_{\epsilon_{k}} * f\right) \|_{2} \\
&=\sup _{M} \sup _{\left|c_{k}\right| \leq 1}\left\{\sum\limits_{\ell \in \mathbb{Z}}\left|\sum\limits_{k=1}^{M} c_{k}\left(\widehat{\varphi}_{\epsilon_{k+1}}(\ell) \widehat{f}(\ell)-\widehat{\varphi}_{\epsilon_{k}}(\ell) \widehat{f}(\ell)\right)\right|^{2}\right\}^{1 / 2} \\
&=\sup _{M} \sup _{\left|c_{k}\right| \leq 1}\left\{\sum\limits_{\ell \in \mathbb{Z}}|\hat{f}(\ell)|^{2}\left|\sum\limits_{k=1}^{M} c_{k}\left(\widehat{\varphi}_{\epsilon_{k+1}}(\ell)-\widehat{\varphi}_{\epsilon_{k}}(\ell)\right)\right|^{2}\right\}^{1 / 2} \\
& \leq\left\{\sum\limits_{\ell \in \mathbb{Z}}|\hat{f}(\ell)|^{2}\left(\sum\limits_{k=1}^{\infty}\left|\widehat{\varphi}_{\epsilon_{k+1}}(\ell)-\widehat{\varphi}_{\epsilon_{k}}(\ell)\right|\right)^{2}\right\}^{1 / 2} \\
& \leq C\left\{\sum\limits_{\ell \in \mathbb{Z}}|\hat{f}(\ell)|^{2}\right\}^{1 / 2}=C\|f\|_{2}
\end{aligned}
$$

Thus we see that the proof is complete.
\end{proof}

It is interesting to note that the above characterization for unconditional convergence in $L^2(\mathbb{T})$ is also the characterization for unconditional convergence in $H_{2}(\mathbb{T})$ so that the following theorem holds true:

\begin{thm}\label{5.2}
If $\epsilon_{k}$ is any sequence of positive reals tending to zero, then following are equivalent:

(a) $\quad \sum\limits_{k=1}^{\infty}\left(D_{\epsilon_{k+1}} f-D_{\epsilon_{k}} f\right)$ is unconditionally convergent in $H_{2}(\mathbb{T})$ for all $f$ in $H_{2}(\mathbb{T})$

(b) There is a finite constant $C$ such that $\sup _{\{\ell \text { in } \mathbb{Z}\}} \sum\limits_{k=1}^{\infty}\left|\widehat{\varphi}_{\epsilon_{k+1}}(\ell)-\widehat{\varphi}_{\epsilon_{k}}(\ell)\right| \leq C$
\end{thm}

A few clarifications of the notation used are made now before the proof of Theorem~\ref{5.2} is presented.  Henceforth in this paper, the reader should interpret $L^2$ and $H_{2}$ to mean $L^2(\mathbb{T})$ and $H_{2}(\mathbb{T})$ respectively. For any $f$ in $H_{2}$, we associate $f^{*}$ in $L^2$ to be the nontangential limit of $f$ such that $f=P\left[f^{*}\right]$. We will also need to recall that $\|f\|_{H_{2}}=\left\|f^{*}\right\|_{L^2}$

It will also be convenient to use the following facts:

\begin{gather*}
\text { All } f \text { in } H_{2} \text { can be expressed as } \sum\limits_{n=0}^{\infty} a_{n} z^{n} \text { with } \sum\limits_{n=0}^{\infty}\left|a_{n}\right|^{2}<\infty  \tag{5.1}\\
\text { If } f(z)=\sum\limits_{n=0}^{\infty} a_{n} z^{n} \text { then } f^{*}\left(e^{i \theta}\right)=\sum\limits_{n=0}^{\infty} a_{n} e^{i n \theta} \tag{5.2}
\end{gather*}

It will also be convenient to use the following:

\begin{lem}\label{5.3}  
If $f \in H_{2}, \rho \leq 1, z=\rho e^{i \theta}$, and

$$
D_{\epsilon} f(z) \equiv \frac{1}{2 \epsilon} \int_{-\epsilon}^{+\epsilon} f\left(z e^{-i x}\right) d x
$$

then

$$
\left(D_{\epsilon} f\right)^{*}\left(e^{i \theta}\right)=D_{\epsilon}\left(f^{*}\right)\left(e^{i \theta}\right)
$$
\end{lem}
\begin{proof}

Since

$$
f(z)=\frac{1}{2 \pi} \int_{-\pi}^{+\pi} \frac{e^{i t}+z}{e^{i t}-z} f^{*}\left(e^{i t}\right) d t
$$

we see that a quick application of Fubini's Theorem gives us:

$$
\begin{aligned}
D_{\epsilon} f(z) & =\frac{1}{2 \pi} \int_{-\pi}^{\pi} \frac{e^{i u}+z}{e^{i u}-z}\left[\frac{1}{2 \epsilon} \int_{-\epsilon}^{+\epsilon} f^{*}\left(e^{i(u-x)}\right) d x\right] d u \\
& =\frac{1}{2 \pi} \int_{-\pi}^{+\pi} \frac{e^{i u}+z}{e^{i u}-z} D_{\epsilon}\left(f^{*}\right)\left(e^{i u}\right) d u
\end{aligned}
$$
\end{proof}

We may now proceed to the proof of Theorem~\ref{5.2}
\begin{proof}[Proof of Theorem~\ref{5.2}]

Suppose we have

$$
\sup _{\ell} \sum\limits_{k=1}^{\infty}\left|\widehat{\phi}_{\epsilon_{k+1}}(\ell)-\widehat{\phi}_{\epsilon_{k}}(\ell)\right|<\infty
$$

Then by Theorem~\ref{5.1}, we know that we have unconditional convergence in $L^2$. That is, we have:

$$
\sup _{M} \sup _{\left|c_{k}\right| \leq 1}\left\|\sum\limits_{k=1}^{M} c_{k}\left(D_{\epsilon_{k+1}}(f)-D_{\epsilon_{k}}(f)\right)\right\|_{2} \leq C\|f\|_{2} \text { for all } f \in L^2
$$

However, this implies that we have:

$$
\sup _{M} \sup _{\left|c_{k}\right| \leq 1}\left\|\sum\limits_{k=1}^{M} c_{k}\left(D_{\epsilon_{k+1}}\left(f^{*}\right)-D_{\epsilon_{k}}\left(f^{*}\right)\right)\right\|_{2} \leq C\left\|f^{*}\right\|_{2} \text { for all } f \in H_{2}
$$

Which, by the lemma gives us:

$$
\sup _{M} \sup _{\left|c_{k}\right| \leq 1}\left\|\sum\limits_{k=1}^{M} c_{k}\left(\left\{D_{\epsilon_{k+1}}(f)\right\}^{*}-\left\{D_{\epsilon_{k}}(f)\right\}^{*}\right)\right\|_{2} \leq C\left\|f^{*}\right\|_{2} \text { for all } f \in H_{2}
$$

Thus since for any $h$ in $H_{2},\|h\|_{H_{2}}=\left\|h^{*}\right\|_{L^2}$, we see that we have:

$$
\sup _{M} \sup _{\left|c_{k}\right| \leq 1}\left\|\sum\limits_{k=1}^{M} c_{k}\left(D_{\epsilon_{k+1}}(f)-D_{\epsilon_{k}}(f)\right)\right\|_{2} \leq C\|f\|_{2} \text { for all } f \in H_{2}
$$

Thus we have unconditional convergence in $\mathrm{H}_{2}$ and this part of the proof is done.

Suppose now that we have $\sum\limits_{k=1}^{\infty}\left(D_{\epsilon_{k+1}} f-D_{\epsilon_{k}} f\right)$ is unconditionally convergent in $H_{2}(\mathbb{T})$ for all $f$ in $H_{2}(\mathbb{T})$

At this point we will need to recall the fact that all $f$ in $H_{2}$ can be expressed as $\sum\limits_{n=0}^{\infty} a_{n} z^{n}$ with $\sum\limits_{n=0}^{\infty}\left|a_{n}\right|^{2}<\infty$

Since we have unconditional convergence in $H_{2}$, Theorem~\ref{2.4} implies that

$$
\sup _{M} \sup _{\left|c_{k}\right| \leq 1}\left\|\sum\limits_{k=1}^{M} c_{k}\left(D_{\epsilon_{k+1}}(f)-D_{\epsilon_{k}}(f)\right)\right\|_{2} \leq C\|f\|_{2} \text { for all } f \text { in } H_{2}
$$

As before, by the lemma, and by the fact that for any $h \in H_{2},\|h\|_{H_{2}}=\left\|h^{*}\right\|_{L^2}$, we see that we have

$$
\sup _{M} \sup _{\left|c_{k}\right| \leq 1}\left\|\sum\limits_{k=1}^{M} c_{k}\left(D_{\epsilon_{k+1}}\left(f^{*}\right)-D_{\epsilon_{k}}\left(f^{*}\right)\right)\right\|_{2} \leq C\left\|f^{*}\right\|_{2}
$$

for all $f^{*}$ of the form $\sum\limits_{n=0}^{\infty} a_{n} e^{i n \theta}$ with $\sum\limits_{n=0}^{\infty}\left|a_{n}\right|^{2}<\infty$.

Proceeding in a fashion similar to that of the proof of Theorem~\ref{5.1}, we fix $\ell_{o}>0$, then we choose $f^{*}$ such that

$$
\widehat{f^{*}}(\ell)= \begin{cases}1 & \text { for } \ell=\ell_{o} \\ 0 & \text { for } \ell \neq \ell_{o}\end{cases}
$$

(Only in this proof, $\ell_{0}$ must be greater than or equal to zero) It follows therefore that

$$
\sup _{\ell \in \mathbb{N}} \sum\limits_{k=1}^{\infty}\left|\widehat{\varphi}_{\epsilon_{k+1}}(\ell)-\widehat{\varphi}_{\epsilon_{k}}(\ell)\right| \leq C
$$

We recall at this point that $\varphi_{\epsilon}(\ell)=\frac{\sin (\ell \epsilon)}{\ell \epsilon}$. Thus we note that $\varphi_{\epsilon}(-\ell)=\varphi_{\epsilon}(\ell)$ from which it follows that

$$
\sup _{\ell \in \mathbb{Z}} \sum\limits_{k=1}^{\infty}\left|\widehat{\varphi}_{\epsilon_{k+1}}(\ell)-\widehat{\varphi}_{\epsilon_{k}}(\ell)\right| \leq C
$$
\end{proof}

In addition to proving the unconditional convergence characterizations for $L^2(\mathbb{T})$ and $\mathrm{H}_{2}(\mathbb{T})$, we will also state (without proof) the characterization for unconditional convergence in $L^2(\mathbb{R})$.

To work in $\mathbb{R}$, we define

$$
\begin{aligned}
D_{\epsilon}^{R} f & =\frac{1}{2 \epsilon} \int_{-\epsilon}^{\epsilon} f(x-y) d y \\
& =\varphi_{\epsilon}^{R} * f
\end{aligned}
$$

where $\varphi_{\epsilon}^{R}$ is the normalized characteristic function of the the interval $(-\epsilon, \epsilon)$ so that $\widehat{\varphi_{\epsilon}^{R}}(t)=$ $\frac{\sin \epsilon t}{\epsilon t}$.

The theorem can now be stated:

\begin{thm}\label{5.4}
If $\epsilon_{k}$ is any sequence of positive reals tending to zero, then following are equivalent:

(a) $\quad \sum\limits_{k=1}^{\infty}\left(D_{\epsilon_{k+1}}^{R} f-D_{\epsilon_{k}}^{R} f\right)$ is unconditionally convergent in $L^2(\mathbb{R})$ for all $f$ in $L^2(\mathbb{R})$

(b) There is a finite constant $C$ such that $\sup _{\{t \text { in } \mathbb{R}\}} \sum\limits_{k=1}^{\infty}\left|\widehat{\varphi_{\epsilon_{k+1}}^{R}}(t)-\widehat{\varphi_{\epsilon_{k}}^{R}}(t)\right| \leq C$
\end{thm}

In addition, the following theorem (in conjuction with Theorem~\ref{5.1} and Theoreme~\ref{5.4}) shows that the criteria for unconditional convergence in $L^2(\mathbb{T})$ is equivalent to the criteria for unconditional convergence in $L^2(\mathbb{R})$ :

\begin{thm}\label{5.5}
If $\epsilon_{k}$ is any sequence of positive reals tending to zero, then the following are equivalent:

(a) There is a finite constant $C$ such that $\sup _{\{\ell \in \mathbb{Z}\}} \sum\limits_{k=1}^{\infty}\left|\widehat{\varphi}_{\epsilon_{k+1}}(\ell)-\widehat{\varphi}_{\epsilon_{k}}(\ell)\right| \leq C$

(b) There is a finite constant $K$ such that $\sup _{\{t \text { in } \mathbb{R}\}} \sum\limits_{k=1}^{\infty}\left|\widehat{\varphi_{\epsilon_{k+1}}^{R}}(t)-\widehat{\varphi_{\epsilon_{k}}^{R}}(t)\right| \leq K$
\end{thm}

One direction of the proof is trivial since the supremum is being taken over a larger set.

The other direction can be proved using the second theorem in Cartwright~\ref{Cartwright}.  Her result can also be found in section 4.3.3 of Timan~\cite{Timan}, page 186. Timan uses the word 'degree' and Cartwright uses the word 'type' to describe the same thing. One of the definitions for type was already described in the comments prior to Theorem~\ref{2.7}. Cartwright also mentions the word 'order'. The order $\rho$ of a function can
either be defined so that $\rho$ is the lower bound of the positive numbers $A$ such that as $|z|=r$ goes to infinity, $f(z)=O\left(e^{r^{A}}\right)$. Thus $e^{z}, e^{5 z}, \sin z$ are all of order one while polynomials are of order zero. Another equivalent definition for the type $\sigma$ is: $\sigma=\lim \sup _{r \rightarrow \infty} \frac{\ln M(r)}{r^{\rho}}$ where $M(r)=\max _{|z|=r}|f(z)|$

At any rate, Cartwright's result can now be stated and understood:

\begin{thm}\label{5.6}  Suppose that $f(z)$ is an integral function of order one and type $k<\pi$, and that

$$
\sup _{\ell \in \mathbb{Z}}|f(n)| \leq A
$$

then

$$
\sup _{x \in \mathbb{R}}|f(x)| \leq K A
$$
\end{thm}

The reader can probably now guess how the proof of the other direction will go.

\begin{proof}[Proof of the nontrivial direction of Theorem~\ref{5.5}]

Suppose that

$$
\sup _{\ell \in \mathbb Z} \sum\limits_{k=1}^{\infty}\left|\widehat{\varphi_{\epsilon_{k+1}}}(\ell)-\widehat{\varphi_{\epsilon_{k}}}(\ell)\right| \leq C
$$

Upon noting that $\varphi_{\epsilon_{k+1}}-\varphi_{\epsilon_{k}}$ is supported in the interval $\left(-\epsilon_{k}, \epsilon_{k}\right)$, and recalling the comments prior to Lemma 2.8 it follows that whenever $M>K_{o}>0$, the sum $\sum\limits_{k=K_{o}+1}^{M}\left(\widehat{\varphi_{k+1}}-\right.$ $\left.\widehat{\varphi_{\epsilon_{k}}}\right)$ will be in $E^{\epsilon_{K_{o}}+1}$.

Let $K_{o}$ be the first $k$ such that $\epsilon_{k}<\pi$, then for all $M>K_{o}>0$,

$$
\begin{aligned}
\sum\limits_{k=1}^{M}\left|\widehat{\varphi_{\epsilon_{k+1}}^{R}}(t)-\widehat{\varphi_{\epsilon_{k}}^{R}}(t)\right| & \leq \sum\limits_{k=1}^{K_{o}}\left|\widehat{\varphi_{\epsilon_{k+1}}^{R}}(t)-\widehat{\varphi_{\epsilon_{k}}^{R}}(t)\right|+\sum\limits_{k=K_{o}+1}^{M}\left|\widehat{\varphi_{\epsilon_{k+1}}^{R}}(t)-\widehat{\varphi_{\epsilon_{k}}^{R}}(t)\right| \\
& \leq \sum\limits_{k=1}^{K_{o}}\left(\left|\widehat{\varphi_{\epsilon_{k+1}}^{R}}(t)\right|+\left|\widehat{\varphi_{\epsilon_{k}}^{R}}(t)\right|\right)+K C \\
& \leq K_{o}\left(\left\|\varphi_{\epsilon_{k+1}}^{R}\right\|_{1}+\left\|\varphi_{\epsilon_{k}}^{R}\right\|_{1}\right)+K C \\
& =2 K_{o}+K C
\end{aligned}
$$

Since $K$ is dependent only on $K_{o}$, (and not on M), the desired result is obtained.
\end{proof}

It will be shown in the two examples below that when $\epsilon_{k}=1 / 2^{k}$ the characterization of unconditional convergence in Theorem~\ref{5.2} and Theorem~\ref{5.1} holds true. This is not a surprise since, this is what happens in the case of the ergodic averages. Furthermore, the characterization for unconditional convergence in Theorem~\ref{5.2} and Theorem~\ref{5.1} fails however when $\epsilon_{k}=1 / k$. Thus there seems to be some sort of analogy between derivatives and averages. (See also and compare with Example~\ref{3.3} and Example~\ref{3.2}.

But before we give the example, we need to prove the following:

\begin{lem}\label{5.6.1}
(i) If $x<y<1$ then $\left|\frac{\sin x}{x}-\frac{\sin y}{y}\right|=\frac{\sin x}{x}-\frac{\sin y}{y}$

(ii) If $\ell / 2^{k}<1$ then $\frac{\sin \ell / 2^{k+1}}{\ell / 2^{k+1}}-\frac{\sin \ell / 2^{k}}{\ell / 2^{k}} \leq \frac{\ell^{2}}{3 ! 4^{k}}$ and
\end{lem}
\begin{proof}

(i) If we let $f(x)=\frac{\sin x}{x}$, we see that $f^{\prime}(x)=\frac{x \cos x-\sin x}{x^{2}}$ But $\operatorname{since} \tan x>x$ for $x<\pi / 2$, we see that $f^{\prime}(x)<0$ for $x<\pi / 2 \ldots$ Since $x<y<1<\pi / 2$, and since $\lim _{x \rightarrow 0} \frac{\sin x}{x}=1$, we see that (i) follows

(ii) If $\theta<1$, we know that $\theta-\theta^{3} / 3 ! \leq \sin \theta \leq \theta$ thus

$$
\frac{\sin \ell / 2^{k+1}}{\ell / 2^{k+1}}-\frac{\sin \ell / 2^{k}}{\ell / 2^{k}} \leq 1-\frac{\ell / 2^{k}-\frac{\ell^{3}}{32^{3 k}}}{\ell / 2^{k}}=\frac{\ell^{2}}{3 ! 4^{k}}
$$
\end{proof}
\medskip

\begin{exmp}~\label{5.7} Here we show that $\epsilon_{k}=1 / 2^{k}$ is good in the sense that unconditional convergence occurs according to the criterion of Theorem~\ref{5.1} and Theorem~\ref{5.2}.

We wish to show that

$$
\sup _{\ell} \sum\limits_{k=1}^{\infty}\left|\widehat{\varphi}_{\epsilon_{k+1}}(\ell)-\widehat{\varphi}_{\epsilon_{k}}(\ell)\right| \leq C
$$

Or, in other words, since

$$
\widehat{\varphi_{\epsilon}}(\ell)=\frac{\sin \ell \epsilon}{\ell \epsilon}
$$

we see that it suffices to show that

$$
\sup _{\ell} \sum\limits_{k=1}^{\infty} d_{k}(\ell)<\infty
$$

where

$$
d_{k}(\ell)=\left|\frac{\sin \ell / 2^{k+1}}{\ell / 2^{k+1}}-\frac{\sin \ell / 2^{k}}{\ell / 2^{k}}\right|
$$

By part (ii) of the claim, we see that

$$
\begin{aligned}
\sum\limits_{k \geq\left[\log _{2} \ell\right]+1} d_{k}(\ell) & \leq \frac{\ell^{2}}{3 !} \sum\limits_{k \geq\left[\log _{2} \ell\right]+1} \frac{1}{4^{k}} 
 =\frac{\ell^{2}}{3 !} \frac{4^{\left[\log _{2} \ell\right]+1}}{1-1 / 4} \\
& =\frac{\ell^{2}}{3 ! 3} \frac{1}{4^{\left[\log _{2} \ell\right]}} 
 \leq \frac{\ell^{2}}{3 ! 3} \frac{1}{2^{2\left(\log _{2} \ell-1\right)}} 
 =\frac{4 \ell^{2}}{3 ! 3} \frac{1}{2^{\log _{2} \ell^{2}}} 
 =2 / 9
\end{aligned}
$$

Furthermore,

$$
\begin{aligned}
\sum\limits_{k \leq\left[\log _{2} \ell\right]} d_{k}(\ell) & \leq \sum\limits_{k \leq\left[\log _{2} \ell\right]} \frac{2}{\ell / 2^{k+1}} 
=\frac{4}{\ell} \sum\limits_{k \leq\left[\log _{2} \ell\right]} 2^{k} \\
& =\frac{4}{\ell} \frac{\left(2^{\left[\log _{2} \ell\right]+1}-2\right)}{2-1} 
 \leq \frac{8}{\ell} 2^{\left[\log _{2} \ell\right]} 
 \leq \frac{8}{\ell} 2^{\log _{2} \ell}=8
\end{aligned}
$$

Thus we see that we're done with the proof that we have unconditional convergence when $\epsilon_{k}=1 / 2^{k}$

\qed
\end{exmp}
\medskip

\begin{exmp}~\label{5.8} Consider the case when $\epsilon_{k}=1 / k$. It turns out that we do indeed have that

$$
\sup _{\ell} \sum\limits_{k=1}^{\infty}\left|\frac{\sin \left(\epsilon_{k+1} \ell\right)}{\epsilon_{k+1} \ell}-\frac{\sin \left(\epsilon_{k} \ell\right)}{\epsilon_{k} \ell}\right|=\infty
$$

and hence we do not have unconditional convergence.

$$
\begin{aligned}
\sup _{\ell} \sum\limits_{k=1}^{\infty}\left|\frac{\sin \left(\epsilon_{k} \ell\right)}{\epsilon_{k} \ell}-\frac{\sin \left(\epsilon_{k+1} \ell\right)}{\epsilon_{k+1} \ell}\right| & =\sup _{\ell} \sum\limits_{k=1}^{\infty}\left|\frac{\sin \left(\frac{\ell}{k}\right)}{\frac{\ell}{k}}-\frac{\sin \left(\frac{\ell}{k+1}\right)}{\frac{\ell}{k+1}}\right| \\
& \geq \sup _{\ell} \sum\limits_{k=\left[\ell^{\alpha}\right]}^{\ell}\left|\frac{\sin \left(\frac{\ell}{k}\right)}{\frac{\ell}{k}}-\frac{\sin \left(\frac{\ell}{k+1}\right)}{\frac{\ell}{k+1}}\right| \quad(\text { where } 0<\alpha<1) \\
& =\sup _{\ell} \sum\limits_{k=\left[\ell^{\alpha}\right]}^{\ell}\left|\frac{k \sin \left(\frac{\ell}{k}\right)-k \sin \left(\frac{\ell}{k+1}\right)}{\ell}-\frac{\sin \left(\frac{\ell}{k+1}\right)}{\ell}\right| \\
& =\sup _{\ell} \sum\limits_{k=\left[\ell^{\alpha}\right]}^{\ell}\left|\frac{k \cos \xi(k, \ell)\left(\frac{\ell}{k}-\frac{\ell}{k+1}\right)}{\ell}-\frac{\sin \left(\frac{\ell}{k+1}\right)}{\ell}\right| \\
& =\sup _{\ell} \sum\limits_{k=\left[\ell^{\alpha}\right]}^{\ell}\left|\frac{\cos (\xi(k, \ell))}{k+1}-\frac{\sin \left(\frac{\ell}{k+1}\right)}{\ell}\right| \\
& \geq \sup _{\ell} \sum\limits_{k=\left[\ell^{\alpha}\right]}^{\ell}\left(\frac{|\cos (\xi(k, \ell))|}{k+1}-\frac{1}{\ell}\right) \\
& =\sup _{\ell}\left\{\left(\sum\limits_{k=\left[\ell^{\alpha}\right]}^{\ell} \frac{|\cos (\xi(k, \ell))|}{k+1}\right)-\frac{\ell-\left[\ell^{\alpha}\right]}{\ell}\right\} \\
& \geq \sup _{\ell} \sum\limits_{k=\left[\ell^{\alpha}\right]}^{\ell} \frac{|\cos (\xi(k, \ell))|}{k+1}-1
\end{aligned}
$$

Thus we see that it suffices to show that

$$
\sup _{\ell} \sum\limits_{k=\left[\ell^{\alpha}\right]}^{\ell} \frac{|\cos (\xi(k, \ell))|}{k+1}=\infty
$$

We note that $|\cos (\xi(k, \ell))|$ is near 1 when $\xi(k, \ell)$ is near some multiple of $\pi$. We now

proceed to compare

$$
\sum\limits_{k=\left[\ell^{\alpha}\right]}^{\ell} \frac{|\cos (\xi(k, \ell))|}{k+1}
$$

with the harmonic series by making the following observations:

There is some positive constant $C$ (which is approximately equal to $1 / \pi$ ) such that there are at least $C\left(\ell /\left[\ell^{\alpha}\right]\right)$ values of $\mathbf{n}$ such that

\begin{equation*}
\frac{\pi}{2}<\pi n-\frac{\pi}{3}<\pi n+\frac{\pi}{3}<\frac{\ell}{\left[\ell^{\alpha}\right]} \tag{5.3}
\end{equation*}

We note that if

$$
\pi n-\frac{\pi}{3}<\frac{\ell}{k+1}<\frac{\ell}{k}<\pi n+\frac{\pi}{3}
$$

then

\begin{equation*}
|\cos (\xi(k, \ell))| \geq \cos \left(\frac{\pi}{3}\right)>0 \tag{5.4}
\end{equation*}

Next, we fix some value of $n$, and estimate the number of times that $\pi n-\frac{\pi}{3}<\xi(k, \ell)<\pi n+\frac{\pi}{3}$ occurs. But for $\pi n-\frac{\pi}{3}<\xi(k, \ell)<\pi n+\frac{\pi}{3}$ to occur, it is necessary that

$$
\pi n-\frac{\pi}{3}<\frac{\ell}{k+1}<\frac{\ell}{k}<\pi n+\frac{\pi}{3} \text { holds. }
$$

Or, in other words, we must have

$$
k+1 \leq \frac{\ell}{\pi n-\frac{\pi}{3}} \quad \text { and } \quad k \geq \frac{\ell}{\pi n+\frac{\pi}{3}}
$$

which implies that there exists some positive constant D such that the number of times that $\pi n-\frac{\pi}{3}<\xi(k, \ell)<\pi n+\frac{\pi}{3}$ occurs is approximately

\begin{align*}
\frac{\ell}{\pi n-\frac{\pi}{3}}-\frac{\ell}{\pi n+\frac{\pi}{3}} & =\frac{\frac{2 \pi}{3} \ell}{(\pi n)^{2}-\frac{\pi^{2}}{9}}  \tag{5.5}\\
& \geq \frac{D \ell}{n^{2}}
\end{align*}

We also notice that when $\pi n-\frac{\pi}{3}<\xi(k, \ell)<\pi n+\frac{\pi}{3}$, it follows from the definition of $\xi(k, \ell)$ that there is some positve constant $\mathrm{F}(\pi / 2$ would work) such that

\begin{equation*}
\frac{1}{k} \geq F \frac{n}{\ell} \text { for any } n, k \text { or } \ell \tag{5.6}
\end{equation*}

Let $U=\left\{n:\left[\frac{\pi}{2}<\pi n-\frac{\pi}{3}<\pi n+\frac{\pi}{3}<\frac{\ell}{\left[\ell^{\ell}\right]}\right\}\right.$ and $V(n)=\left\{k: \pi n-\frac{\pi}{3}<\xi(k, \ell)<\pi n+\frac{\pi}{3}\right\}$

We see that

$$
\begin{aligned}
\sum\limits_{k=\left[\ell^{\alpha}\right]}^{\ell} \frac{|\cos (\xi(k, \ell))|}{k+1} & \geq \sum\limits_{\{n: n \text { is in } U\}} \sum\limits_{\{k: k \text { is in } V(n)\}} \frac{\cos (\xi(k, l))}{k} \\
& \geq \cos \left(\frac{\pi}{3}\right) \sum\limits_{\{n: n \text { is in } U\}\{k: k \text { is in } V(n)\}} \frac{1}{k} \\
& \geq \cos \left(\frac{\pi}{3}\right) \sum\limits_{\{n: n \text { is in } U\}\{k: k \text { is in } V(n)\}} F \frac{n}{\ell}(\text { by inequality (5.6)) } \\
& \geq \cos \left(\frac{\pi}{3}\right) \sum\limits_{\{n: n \text { is in } U\}}\left(\frac{D \ell}{n^{2}}\right)\left(\frac{F n}{\ell}\right) \\
& =D F \cos \left(\frac{\pi}{3}\right) \sum\limits_{\{n: n \text { is in } U\}} \frac{1}{n} \\
& =D F \cos \left(\frac{\pi}{3}\right) \sum\limits_{n=1}^{C \ell /\left[\ell^{\alpha}\right]} \frac{1}{n}
\end{aligned}
$$

which is greater than or equal to

$$
D F \cos \left(\frac{\pi}{3}\right) \ln \ell^{1-\alpha}
$$

Thus taking the supremum over all $\ell$ gives the desired result.
\qed
\end{exmp}
\medskip

It turns out that if we replace $\varphi_{\epsilon}$ by something smoother at the edges (namely $\psi_{\epsilon}=$ $\varphi_{\epsilon} * \varphi_{\epsilon}$ ) then we have unconditional convergence in Theorem~\ref{5.1} and Theorem~\ref{5.2} for all $\epsilon_{k}$ going to zero.

But before stating and proving the theorem, we note that:

$$
\begin{aligned}
\left(\varphi_{\epsilon} * \varphi_{\epsilon}\right)\left(e^{i x}\right) & =\frac{1}{4 \epsilon^{2}} \int_{-\epsilon}^{+\epsilon} 1_{[-\epsilon, \epsilon]}(x-y) d y \\
& =\frac{m([x-\epsilon, x+\epsilon] \cap[-\epsilon,+\epsilon])}{4 \epsilon^{2}} \\
& = \begin{cases}\frac{1}{4 \epsilon^{2}} x+\frac{1}{2 \epsilon} & \text { when }-2 \epsilon<x \leq 0 \\
\frac{-1}{4 \epsilon^{2}} x+\frac{1}{2 \epsilon} & \text { when } 0 \leq x<2 \epsilon \\
0 & \text { otherwise }\end{cases}
\end{aligned}
$$

The theorem is as follows

\begin{thm}\label{5.9} If $\epsilon_{1} \geq \epsilon_{2} \geq \epsilon_{3} \ldots$, with $\epsilon_{k}$ approaching $0$, then

$$
\sup _{\ell} \sum\limits_{k=1}^{\infty}\left|\widehat{\psi}_{\epsilon_{k}}(\ell)-\widehat{\psi}_{\epsilon_{k+1}}(\ell)\right|<\infty
$$
\end{thm}
\begin{proof}
First, we claim that it suffices to prove the result for the case when $\epsilon_{k}$ approaches $0$ slowly, or more specifically, when

\begin{equation*}
\frac{\epsilon_{k}}{\epsilon_{k+1}} \leq 2 \tag{5.7}
\end{equation*}

Suppose that $d_{j}$ is some sequence which goes to zero very quickly, we can always construct some sequence $\epsilon_{k}$ which approaches zero, such that (5.7) and $d_{j}=\epsilon_{s_{j}}$ hold true so that $d_{j}$ is a subsequence of the slower sequence $\epsilon_{k}$. We then have that

$$
\begin{aligned}
\left|\widehat{\psi}_{d_{j}}(\ell)-\widehat{\psi}_{d_{j+1}}(\ell)\right| & =\left|\widehat{\psi}_{\epsilon_{s_{j}}}(\ell)-\widehat{\psi}_{\epsilon_{s_{j+1}}}(\ell)\right| \\
& \leq \sum\limits_{k=s_{j}}^{s_{j+1}-1}\left|\widehat{\psi}_{\epsilon_{k}}(\ell)-\widehat{\psi}_{\epsilon_{k+1}}(\ell)\right|
\end{aligned}
$$

Summing over all $j$ 's we see that the claim follows. Thus without loss of generality we assume that (5.7) holds and we continue with the proof by noting that:

$$
\begin{aligned}
\sum\limits_{k=1}^{\infty}\left|\widehat{\psi}_{\epsilon_{k}}(\ell)-\widehat{\psi}_{\epsilon_{k+1}}(\ell)\right| & =\sum\limits_{k=1}^{\infty}\left|\widehat{\varphi}_{\epsilon_{k}}^{2}(\ell)-\widehat{\varphi}_{\epsilon_{k+1}}^{2}(\ell)\right| \\
& =\sum\limits_{k=1}^{\infty}\left|\widehat{\varphi}_{\epsilon_{k}}(\ell)+\widehat{\varphi}_{\epsilon_{k+1}}(\ell)\right|\left|\widehat{\varphi}_{\epsilon_{k}}(\ell)-\widehat{\varphi}_{\epsilon_{k+1}}(\ell)\right| \\
& =\sum\limits_{1}+\sum\limits_{2}
\end{aligned}
$$

where

$$
\sum\limits_{1}=\sum\limits_{k:\left|\epsilon_{k} \ell\right| \leq 1}\left|\widehat{\varphi}_{\epsilon_{k}}(\ell)+\widehat{\varphi}_{\epsilon_{k+1}}(\ell)\right|\left|\widehat{\varphi}_{\epsilon_{k}}(\ell)-\widehat{\varphi}_{\epsilon_{k+1}}(\ell)\right|
$$

and

$$
\sum\limits_{2}=\sum\limits_{k:\left|\epsilon_{k} \ell\right|>1}\left|\widehat{\varphi}_{\epsilon_{k}}(\ell)+\widehat{\varphi}_{\epsilon_{k+1}}(\ell)\right|\left|\widehat{\varphi}_{\epsilon_{k}}(\ell)-\widehat{\varphi}_{\epsilon_{k+1}}(\ell)\right|
$$

Fix $\ell$. Since $\widehat{\varphi}_{\epsilon_{k}}(\ell)$ is an even function of $\ell$, we may assume without loss of generality that $\ell \geq 0$. Furthermore, since $\frac{\sin x}{x}$ is decreasing with respect to $\mathrm{x}$ on $[0, \pi)$, it follows that:

$$
\begin{aligned}
\sum\limits_{1} & =\sum\limits_{k:\left|\epsilon_{k} \ell\right| \leq 1}\left|\widehat{\varphi}_{\epsilon_{k}}(\ell)+\widehat{\varphi}_{\epsilon_{k+1}}(\ell)\right|\left|\widehat{\varphi}_{\epsilon_{k}}(\ell)-\widehat{\varphi}_{\epsilon_{k+1}}(\ell)\right| \\
& \leq 2 \sum\limits_{k:\left|\epsilon_{k} \ell\right| \leq 1}\left|\widehat{\varphi}_{\epsilon_{k+1}}(\ell)-\widehat{\varphi}_{\epsilon_{k}}(\ell)\right| \\
& =2 \sum\limits_{k:\left|\epsilon_{k} \ell\right| \leq 1}\left|\frac{\sin \epsilon_{k+1} \ell}{\epsilon_{k+1} \ell}-\frac{\sin \epsilon_{k} \ell}{\epsilon_{k} \ell}\right| \\
& =2 \sum\limits_{k:\left|\epsilon_{k} \ell\right| \leq 1} \frac{\sin \epsilon_{k+1} \ell}{\epsilon_{k+1} \ell}-\frac{\sin \epsilon_{k} \ell}{\epsilon_{k} \ell} \\
& =2\left(1-\frac{\sin \epsilon_{k_{o}} \ell}{\epsilon_{k_{o}} \ell}\right)
\end{aligned}
$$

where $k_{o}$ is the first $k$ such that $\epsilon_{k_{o}} \ell \leq 1$. Furthermore, $\epsilon_{k_{o}} \ell \leq 1$ implies that $\sin 1 \leq \frac{\sin \epsilon_{k_{o}} \ell}{\epsilon_{k_{o}} \ell}$, thus $\sum\limits_{1} \leq 2(1-\sin 1)$ Thus the estimate on $\sum\limits_{2}$ is done.

In what follows next, in computing the estimate on $\sum\limits_{2}$, is where we'll use the assumption (5.7) which we mentioned earlier.

Once again we note that we can assume $\ell>0$ since $\widehat{\varphi}$ is an even function of $\ell$. Thus

$$
\begin{aligned}
\sum\limits_{2} & =\sum\limits_{k:\left|\epsilon_{k} \ell\right|>1}\left|\widehat{\varphi}_{\epsilon_{k}}(\ell)+\widehat{\varphi}_{\epsilon_{k+1}}(\ell)\right|\left|\widehat{\varphi}_{\epsilon_{k}}(\ell)-\widehat{\varphi}_{\epsilon_{k+1}}(\ell)\right| \\
& =\sum\limits_{k: \epsilon_{k} \ell>1}\left|\frac{\sin \epsilon_{k+1} \ell}{\epsilon_{k+1} \ell}-\frac{\sin \epsilon_{k} \ell}{\epsilon_{k} \ell}\right|\left|\frac{\sin \epsilon_{k+1} \ell}{\epsilon_{k+1} \ell}+\frac{\sin \epsilon_{k} \ell}{\epsilon_{k} \ell}\right| \\
& \leq \sum\limits_{k: \epsilon_{k} \ell>1} \frac{2}{\epsilon_{k+1} \ell}\left|\frac{\xi \cos \xi-\sin \xi}{\xi^{2}}\left(\epsilon_{k+1} \ell-\epsilon_{k} \ell\right)\right|
\end{aligned}
$$

where $\epsilon_{k+1} \ell \leq \xi \leq \epsilon_{k} \ell$ by the Mean Value Theorem.

But since

$$
\begin{aligned}
\left|\frac{\xi \cos \xi-\sin \xi}{\xi^{2}}\left(\epsilon_{k+1} \ell-\epsilon_{k} \ell\right)\right| & \leq \frac{\xi+1}{\xi^{2}} \ell\left(\epsilon_{k}-\epsilon_{k+1}\right) \\
& \leq \frac{\epsilon_{k} \ell+1}{\epsilon_{k+1} \ell^{2}} \ell\left(\epsilon_{k}-\epsilon_{k+1}\right)
\end{aligned}
$$

$$
\begin{aligned}
& \leq \frac{2 \epsilon_{k} \ell}{\epsilon_{k+1} \ell^{2}} \ell\left(\epsilon_{k}-\epsilon_{k+1}\right) \\
& =\frac{2 \epsilon_{k}}{\epsilon_{k+1}{ }^{2}}\left(\epsilon_{k}-\epsilon_{k+1}\right)
\end{aligned}
$$

we see that

$$
\begin{aligned}
\sum\limits_{2} & \leq \sum\limits_{k: \epsilon_{k} \ell>1} \frac{2}{\epsilon_{k+1} \ell} \frac{2 \epsilon_{k}}{\epsilon_{k+1}{ }^{2}}\left(\epsilon_{k}-\epsilon_{k+1}\right) \\
& =\sum\limits_{k: \epsilon_{k} \ell>1} \frac{4 \epsilon_{k}^{2}}{\epsilon_{k+1}^{2} \ell}\left(\frac{1}{\epsilon_{k+1}}-\frac{1}{\epsilon_{k}}\right) \\
& \leq \sum\limits_{k: \epsilon_{k} \ell>1} \frac{16}{\ell}\left(\frac{1}{\epsilon_{k+1}}-\frac{1}{\epsilon_{k}}\right) \\
& =\frac{16}{\ell}\left(\frac{1}{\epsilon_{k_{o}+1}}-\frac{1}{\epsilon_{1}}\right) \\
& \leq \frac{16}{\epsilon_{k_{o}+1} \ell} \leq 16
\end{aligned}
$$
\end{proof}

\section{Some Interpolation}\label{6}

There will be two theorems presented here. Both theorems could be proved using either of two methods. The first method involves using the method of Benedek, Calderon, and Panzone~\cite{BCP}. The second method involves using BMO estimates. In order to illustrate both methods we choose to prove Theorem~\ref{6.1} using the method of Benedek, Calderon, and Panzone~\cite{BCP} while proving Theorem~\ref{6.2} using BMO estimates.

As before, in the theorem below,

$$
D_{\epsilon} f=\varphi_{\epsilon} * f, \text { where } \varphi_{\epsilon}(x)=\frac{2}{\epsilon} 1_{[-\epsilon, \epsilon]}(x)
$$

\begin{thm}\label{6.1} If $\epsilon_{k+1} \leq r \epsilon_{k}$ where $r<1$, then for $1<p<\infty$, and for any $f$ in $L^p$, the series $\sum\limits_{k=1}^{\infty} D_{\epsilon_{k+1}} f-D_{\epsilon_{k}} f$ is unconditionally convergent in $L^p$.
\end{thm}
\begin{proof}
Recall that this means that we wish to show that for all $p, 1<p<\infty$, there is a constant $c_{p}$ such that

$$
\sup _{M \geq 1} \sup _{\left|c_{k}\right| \leq 1}\left\|\sum\limits_{k=1}^{M} c_{k}\left(D_{\epsilon_{k+1}} f-D_{\epsilon_{k}} f\right)\right\|_{p} \leq c_{p}\|f\|_{p}
$$

The $L^2$ version of this has already been shown in Example\ref{5.7}.

The reader might wish to look at Theorem 2.1 of Jones, Ostrovskii, and Rosenblatt~\cite{JOR} where a weak $L^1$ inequality is proved using the same method as what we are about to use. By looking at the proofs of Theorems 1 and 2 in Benedek, Calderon, and Panzone~\cite{BCP}, it can be seen that it suffices to show that the following Holder estimate holds true: For all $\mathrm{y}$,

$$
\int_{|x|>4|y|}\|K(x-y)-K(x)\|_{\ell_{2}} d x \leq C
$$

where $\mathrm{C}$ is independent of $y$, the $c_{k}$ 's and of $M$ and where $K(x)=\left[c_{k}\left(\varphi_{\epsilon_{k+1}}(x)-\varphi_{\epsilon_{k}}(x)\right)\right]_{k=1}^{M}$

We note that

$$
k(x-y)-k(x)=\left[c_{k}\left(\varphi_{\epsilon_{k+1}}-\varphi_{\epsilon_{k+1}}(x)+\varphi_{\epsilon_{k}}(x)-\varphi_{\epsilon_{k}}(x-y)\right)\right]_{k=1}^{M}
$$

and

$$
\left|\varphi_{\epsilon_{k}}(x-y)-\varphi_{\epsilon_{k}}(x)\right|=\frac{2}{\epsilon_{k}} h\left(\epsilon_{k}, x, y\right)
$$

where

$$
h\left(\epsilon_{k}, x, y\right)= \begin{cases}1_{\left[-\epsilon_{k}, \epsilon_{k}\right]}(x)+1_{\left[y-\epsilon_{k}, y+\epsilon_{k}\right]}(x) & \text { when } 2 \epsilon_{k}<y \\ 1_{\left[-\epsilon_{k}, y-\epsilon_{k}\right]}(x)+1_{\left[\epsilon_{k}, y+\epsilon_{k}\right]}(x) & \text { when } 0<y<2 \epsilon_{k} \\ 1_{\left[y-\epsilon_{k},-\epsilon_{k}\right]}(x)+1_{\left[y+\epsilon_{k}, \epsilon_{k}\right]}(x) & \text { when }-2 \epsilon_{k}<y<0 \\ 1_{\left[y-\epsilon_{k}, y+\epsilon_{k}\right]}+1_{\left[-\epsilon_{k}, \epsilon_{k}\right]}(x) & \text { when } y<-2 \epsilon_{k}\end{cases}
$$

We break up $|x|>4|y|$ into four regions: Region I: $x>0, y>0$, Region II: $x<0, y>0$, Region III: $x<0, y<0$, Region IV: $x>0, y<0$.

Region I: We note that in Region I, we have $x>0, y>0$ and $x>4 y$ We observe that when $y>2 \epsilon_{k}$ then $x>8 \epsilon_{k}$. Thus
$
\left|\varphi_{\epsilon_{k}}(x-y)-\varphi_{\epsilon_{k}}(x)\right|=0 \text { when } y>2 \epsilon_{k}.
$
Furthermore, when $2 \epsilon_{k}>y>\epsilon_{k}$, it follows that $x>4 \epsilon_{k}>y+\epsilon_{k}$, thus
$
\left|\varphi_{\epsilon_{k}}(x-y)-\varphi_{\epsilon_{k}}(x)\right|=0 \text { when } 2 \epsilon_{k}>y>\epsilon_{k}.
$
If $0<y \leq \epsilon_{k}$, then $y-\epsilon_{k}<0<x$, so
$
\left|\varphi_{\epsilon_{k}}(x-y)-\varphi_{\epsilon_{k}}(x)\right|=\frac{2}{\epsilon_{k}} 1_{\left[\epsilon_{k}, y+\epsilon_{k}\right]}(x) \text { when } 0<y \leq \epsilon_{k}.
$

Fix $y$, then

$$
\begin{aligned}
& \int_{\text {Region I }}\|K(x-y)-K(x)\|_{\ell_{2}} d x \leq \\
& \int_{\text {Region I }}\left(\sum\limits_{k=1}^{\infty}\left|c_{k}\right|^{2}\left|\varphi_{\epsilon_{k+1}}(x-y)-\varphi_{\epsilon_{k+1}}(x)\right|^{2}\right)^{1 / 2} d x+ \\
& \int_{\text {Region I }}\left(\sum\limits_{k=1}^{\infty}\left|c_{k}\right|^{2}\left|\varphi_{\epsilon_{k}}(x-y)-\varphi_{\epsilon_{k}}(x)\right|^{2}\right)^{1 / 2} d x=
\end{aligned}
$$

$$
\begin{aligned}
& \int_{\text {Region I }}\left(\sum\limits_{k: y \leq \epsilon_{k+1}}\left|c_{k}\right|^{2}\left(\frac{2}{\epsilon_{k+1}}\right)^{2} 1_{\left[\epsilon_{k+1}, y+\epsilon_{k+1}\right]}(x)\right)^{1 / 2} d x+ \\
& \int_{\text {Region I }}\left(\sum\limits_{k: y \leq \epsilon_{k}}\left|c_{k}\right|^{2}\left(\frac{2}{\epsilon_{k}}\right)^{2} 1_{\left[\epsilon_{k}, y+\epsilon_{k}\right]}(x)\right)^{1 / 2} d x \leq \\
& \int_{\text {Region I }}\left(\sum\limits_{k: y \leq \epsilon_{k+1}}\left(\frac{2}{\epsilon_{k+1}}\right)^{2} 1_{\left[\epsilon_{k+1}, y+\epsilon_{k+1}\right]}(x)\right)^{1 / 2} d x+ \\
& \int_{\text {Region I }}\left(\sum\limits_{k: y \leq \epsilon_{k}}\left(\frac{2}{\epsilon_{k}}\right)^{2} 1_{\left[\epsilon_{k}, y+\epsilon_{k}\right]}(x)\right)^{1 / 2} d x= \\
& 2 \int_{\text {Region I }}\left(\sum\limits_{k: y \leq \epsilon_{k}}\left(\frac{2}{\epsilon_{k}}\right)^{2} 1_{\left[\epsilon_{k}, y+\epsilon_{k}\right]}(x)\right)^{1 / 2} d x \\
& \leq 2 \int_{0}^{\infty} \sum\limits_{k: y \leq \epsilon_{k}}\left(\frac{2}{\epsilon_{k}}\right) 1_{\left[\epsilon_{k}, y+\epsilon_{k}\right]}(x) d x \\
& =4 y \sum\limits_{k: y \leq \epsilon_{k}}\left(\frac{1}{\epsilon_{k}}\right)
\end{aligned}
$$

Let $k_{o}+1$ be the first $k$ such that $\epsilon_{k} \leq y$. Then

$$
\begin{aligned}
4 y \sum\limits_{k: y \leq \epsilon_{k}}\left(\frac{1}{\epsilon_{k}}\right) & =4 y \sum\limits_{k=1}^{k_{o}} \frac{1}{\epsilon_{k}}=4 y \sum\limits_{j=0}^{k_{o}-1} \frac{1}{\epsilon_{k_{o}-j}} \\
& \leq 4 y \sum\limits_{j=0}^{\infty} \frac{r^{j}}{\epsilon_{k_{o}}} \leq 4 \sum\limits_{j=0}^{\infty} r^{j}<\infty
\end{aligned}
$$

It is left to the reader to show that

$$
\int_{\text {Region a }}\|K(x-y)-K(x)\|_{\ell_{2}} d x \leq 4 \sum\limits_{j=0}^{\infty} r^{j}
$$

where "a" can be replaced by II, III or IV. Combining the results from the four regions gives:

$$
\int_{|x|>4|y|}\|K(x-y)-K(x)\|_{\ell_{2}} d x \leq 16 \sum\limits_{j=0}^{\infty} r^{j}<\infty
$$

Thus the desired result is attained.
\end{proof}

We now show that the averaging version of Theorem~\ref{6.1} also holds true. To make the appropriate norm estimate it is convenient to look at the question of whether there is a bound in BMO for the partial sums of the series of differences, given that the function is in $L_{\infty}$. It is not clear whether or not this method will be successful for any sequence $n_{k}$, but certainly in the case that the sequence $\left(n_{k}\right)$ is lacunary as in Example~\ref{3.3}, this method of interpolation does indeed work.

\begin{thm}\label{6.2}  If $\left(n_{k}\right)$ is a lacunary sequence, then for any dynamical system, any $p, 1<p<\infty$, and any $f \in L^p(X)$, the series $\sum\limits_{k=1}^{\infty} M_{n_{k+1}} f-M_{n_{k}} f$ is unconditionally convergent in $L^p(X)$.
\end{thm}
\begin{proof}
Because of duality, it suffices to prove the result in $L^p(X)$ where $2 \leq p<\infty$. Indeed, using Theorem~\ref{2.4}, this theorem is proved if one shows that there is a constant $C_{p}$ depending only on $p$ such that for all dynamical systems and all $f \in L^p(X)$,

$$
\sup _{M \geq 1}\sup_{\left|c_{k}\right| \leq 1} \left\|\sum\limits_{k=1}^{M}c_k\left(M_{n_{k+1}} f-M_{n_{k}} f\right)\right\|_{p} \leq C_{p}\|f\|_{p}
$$

Suppose that this holds for $2 \leq p<\infty$. Then let $1<q \leq 2, h \in L^{q}(X)$, and take $p$ and $q$ to be conjugate indices. Denote by $<,>$ the usual pairing of $L^p(X)$ and $L^{q}(X)$ as dual spaces. Then we have

$$
\begin{aligned}
\left|<\sum\limits_{k=1}^{M} c_{k}\left(M_{n_{k+1}}^{\tau} h-M_{n_{k}}^{\tau} h\right), f>\right| & =\left|<h, \sum\limits_{k=1}^{M} \overline{c_{k}}\left(M_{n_{k+1}}^{\tau^{-1}} f-M_{n_{k}}^{\tau^{-1}} f\right)>\right| \\
& \leq\|h\|_{q}\left\|\sum\limits_{k=1}^{M} \overline{c_{k}}\left(M_{n_{k+1}}^{\tau^{-1}} f-M_{n_{k}}^{\tau^{-1}} f\right)\right\|_{p} \\
& \leq\|h\|_{q} C_{p}\|f\|_{p} .
\end{aligned}
$$

But then $\left\|\sum\limits_{k=1}^{M} c_{k}\left(M_{n_{k+1}}^{\tau} h-M_{n_{k}}^{\tau} h\right)\right\|_{q} \leq C_{p}\|h\|_{q}$ for all $h \in L^{q}(X)$. This proves the result for all $p, 1<p<\infty$.

But if $M \geq 1$ and $\left(c_{k}\right),\left|c_{k}\right| \leq 1$ are fixed, we can consider the linear operator $S f=$ $\sum\limits_{k=1}^{M} c_{k}\left(M_{n_{k+1}} f-M_{n_{k}} f\right)$. By the lacunarity of $\left(n_{k}\right)$ and the argument in Example~\ref{3.3}, we have

$$
\|S f\|_{2} \leq C_{2}\|f\|_{2}
$$

for all $f \in L^2(X)$, with a constant $C_{2}$ which does not depend on $M$ or $\left(c_{k}\right)$. But it is also the case that there is a constant $C_{\infty}$, not depending on $M$ or $\left(c_{k}\right)$, such that

$$
\|S f\|_{B M O} \leq C_{\infty}\|f\|_{\infty} .
$$

Using the usual interpolation theorem between $L^2$ and BMO gives then for each $p, 2 \leq$ $p<\infty$, a constant $C_{p}$ depending only on $C_{2}$ and $C_{\infty}$ such that for all $f \in L^p(X)$,

$$
\|S f\|_{p} \leq C_{p}\|f\|_{p}
$$

This is the criterion that was needed to apply Theorem~\ref{2.4} in proving unconditional convergence. See Bennett and Sharpley~\cite{BS} or Rivi\'ere~\cite{Riviere} for the appropriate interpolation theorem used here.

This argument has thus reduced the proof of Theorem~\ref{6.2} to proving a BMO estimate. The same argument as the ones used in Jones, Kaufman, Rosenblatt and Wierdl~\cite{JKRW} applies here. For simplicity we assume that $\left(n_{k}\right)$ is sufficiently lacunary for it to be the case that $\frac{n_{k+1}}{n_{k}}$ is at least 2 . This simplification is seen by writing the general lacunary sequence as a finite union of such sufficiently lacunary sequences. Also, as is typical, the norm estimate will be computed in $\mathbb{Z}$ and then transferred to the dynamical system.

Without loss of generality, we may take $f \in l_{\infty}(\mathbb{Z})$ to be positive. Fix an interval $I$ in $\mathbb{Z}$ and let $\tilde{I}$ denote an interval with the same left endpoint as $I$ but with three times the length $|I|$ of the interval $I$. Let $f_{1}(x)=f(x) \chi_{\tilde{I}}(x)$ and let $f_{2}(x)=f(x)-f_{1}(x)$. In what follows, let $x$ and $y$ be points in the interval $I$. By the linearity of $S$ it is clear that

$$
|S f(x)-S f(y)| \leq\left|S f_{1}(x)-S f_{1}(y)\right|+\left|S f_{2}(x)-S f_{2}(y)\right| .
$$

We need to show that

$$
\frac{1}{|I|} \sum\limits_{x \in I}\left|S f(x)-\frac{1}{|I|} \sum\limits_{y \in I} S f(y)\right| \leq C_{\infty}\|f\|_{\infty}
$$

We have

$$
\begin{aligned}
\frac{1}{|I|} \sum\limits_{x \in I}\left|S f(x)-\frac{1}{|I|} \sum\limits_{y \in I} S f(y)\right| & =\frac{1}{|I|} \sum\limits_{x \in I}\left|\frac{1}{|I|} \sum\limits_{y \in I} S f(x)-S f(y)\right| \\
\leq & \frac{1}{|I|^{2}} \sum\limits_{x \in I} \sum\limits_{y \in I}|S f(x)-S f(y)| \\
\leq & \frac{1}{|I|^{2}} \sum\limits_{x \in I} \sum\limits_{y \in I}\left|S f_{1}(x)-S f_{1}(y)\right| \\
& +\frac{1}{|I|^{2}} \sum\limits_{x \in I} \sum\limits_{y \in I}\left|S f_{2}(x)-S f_{2}(y)\right| \\
= & A+B
\end{aligned}
$$

We first study $A$. We have

$$
\begin{aligned}
A & =\frac{1}{|I|^{2}} \sum\limits_{x \in I} \sum\limits_{y \in I}\left|S f_{1}(x)-S f 1_{1}(y)\right| \\
& \leq \frac{1}{|I|^{2}} \sum\limits_{x \in I} \sum\limits_{y \in I} S f_{1}(x)+S f_{1}(y) \\
& =\frac{1}{|I|^{2}} \sum\limits_{x \in I} \sum\limits_{y \in I} S f_{1}(x)+\frac{1}{|I|^{2}} \sum\limits_{x \in I} \sum\limits_{y \in I} S f_{1}(y)
\end{aligned}
$$

Using H\"older's inequality, this becomes

$$
\begin{aligned}
A & \leq \frac{1}{|I|^{2}} \sum\limits_{y \in I} \sqrt{|I|}\left(\sum\limits_{x \in I} S f_{1}(x)^{2}\right)^{\frac{1}{2}}+\frac{1}{|I|^{2}} \sum\limits_{x \in I} \sqrt{|I|}\left(\sum\limits_{y \in I} S f_{1}(y)^{2}\right)^{\frac{1}{2}} \\
& \leq \frac{1}{|I|^{2}} \sum\limits_{y \in I} \sqrt{|I|}\left\|S f_{1}\right\|_{2}+\frac{1}{|I|^{2}} \sum\limits_{x \in I} \sqrt{|I|}\left\|S f_{1}\right\|_{2} \\
& \leq \frac{1}{|I|^{2}} \sum\limits_{y \in I} \sqrt{|I|} C_{2}\left\|f_{1}\right\|_{2}+\frac{1}{|I|^{2}} \sum\limits_{x \in I} \sqrt{|I|} C_{2}\left\|f_{1}\right\|_{2} \\
& \leq \frac{1}{|I|^{2}} \sum\limits_{y \in I} \sqrt{|I|} C_{2}\left\|f_{1}\right\|_{\infty} \sqrt{3|I|}+\frac{1}{|I|^{2}} \sum\limits_{x \in I} \sqrt{|I|} C_{2}\left\|f_{1}\right\|_{\infty} \sqrt{3|I|} \\
& \leq \tilde{C}_{\infty}\|f\|_{\infty} .
\end{aligned}
$$

where the constant $\tilde{C}_{\infty}$ only depends on $C_{2}$.

We now need to consider $B$ which was $\frac{1}{|I|^{2}} \sum\limits_{x \in I} \sum\limits_{y \in I}\left|S f_{2}(x)-S f_{2}(y)\right|$. First note that for $n_{k} \leq|I|$ we have $M_{n_{k}} f_{2}(z)=0$ for all $z \in I$, and hence the same is true for all but one term of the series of differences $M_{n_{k+1}} f-M_{n_{k}} f_{2}(z)$ (because for those $n_{k}$ we do not reach the support of the function $f_{2}$.) Let $k_{0}$ denote the smallest $k$ such that $n_{k}>|I|$.

Hence by the remark above,

$$
\begin{aligned}
& B \leq \frac{1}{|I|^{2}} \sum\limits_{x \in I} \sum\limits_{y \in I}\left(\left|M_{n_{k_{0}}} f_{2}(x)-M_{n_{k_{0}}} f_{2}(y)\right|+\right. \\
& \left.\quad+\sum\limits_{k=k_{0}}^{M}\left|M_{n_{k}} f_{2}(x)-M_{n_{k+1}} f_{2}(x)-\left(M_{n_{k}} f_{2}(y)-M_{n_{k+1}} f_{2}(y)\right)\right|\right) \\
& =\frac{1}{|I|^{2}} \sum\limits_{x \in I} \sum\limits_{y \in I}\left|M_{n_{k_{0}}} f_{2}(x)-M_{n_{k_{0}}} f_{2}(y)\right|+ \\
& \quad+\frac{1}{|I|^{2}} \sum\limits_{x \in I} \sum\limits_{y \in I} \sum\limits_{k=k_{0}}^{M}\left|M_{n_{k}} f_{2}(x)-M_{n_{k+1}} f_{2}(x)-\left(M_{n_{k}} f_{2}(y)-M_{n_{k+1}} f_{2}(y)\right)\right| \\
& =B_{1}+B_{2} .
\end{aligned}
$$

$$
\begin{gathered}
\left(\sum\limits_{k=k_{0}}^{M} \left\lvert\,\left(\frac{1}{n_{k}}-\frac{1}{n_{k+1}}\right) \sum\limits_{t=x+1}^{x+n_{k}} f_{2}(t)-\frac{1}{n_{k+1}} \sum\limits_{t=x+n_{k}+1}^{x+n_{k+1}} f_{2}(t)-\right.\right. \\
\left.\left.-\left(\left(\frac{1}{n_{k}}-\frac{1}{n_{k+1}}\right) \sum\limits_{t=y+1}^{y+n_{k}} f_{2}(t)-\frac{1}{n_{k+1}} \sum\limits_{t=y+n_{k}+1}^{y+n_{k+1}} f_{2}(t)\right) \right\rvert\,\right) \\
\leq \frac{1}{|I|^{2}} \sum\limits_{x \in I} \sum\limits_{y \in I}\left(\sum\limits_{k=k_{0}}^{M}\left|\left(\frac{1}{n_{k}}-\frac{1}{n_{k+1}}\right)\left(\sum\limits_{t=x+1}^{x+n_{k}} f_{2}(t)-\sum\limits_{t=y+1}^{y+n_{k}} f_{2}(t)\right)\right|\right)+ \\
+\frac{1}{|I|^{2}} \sum\limits_{x \in I} \sum\limits_{y \in I}\left(\sum\limits_{k=k_{0}}^{M}\left|\frac{1}{n_{k+1}}\left(\sum\limits_{t=x+n_{k}+1}^{x+n_{k+1}} f_{2}(t)-\sum\limits_{t=y+n_{k}+1}^{y+n_{k+1}} f_{2}(t)\right)\right|\right)
\end{gathered}
$$

$$
=B_{3}+B_{4} \text {. }
$$

Here we can estimate $B_{3}$ as follows. Consider the two cases, $x \geq y$ and $y \geq x$ separately. Assume first that $x \geq y$. Then because $k \geq k_{0}, y+n_{k} \geq x+1$ for $x, y \in I$. Hence, the difference $\left|\sum\limits_{t=x+1}^{x+n_{k}} f_{2}(t)-\sum\limits_{t=y+1}^{y+n_{k}} f_{2}(t)\right|$ is no more than $2\|f\|_{\infty}|I|$. A similar estimate holds if $y \geq x$. Thus, in any case we have

$$
\begin{aligned}
B_{3} & \leq \frac{1}{|I|^{2}} \sum\limits_{x \in I} \sum\limits_{y \in I}\left(\sum\limits_{k=k_{0}}^{M}\left(\frac{1}{n_{k}}-\frac{1}{n_{k+1}}\right) 2\|f\|_{\infty}|I|\right) \\
& \leq 2\|f\|_{\infty} \frac{1}{|I|} \sum\limits_{x \in I} \sum\limits_{y \in I}\left(\sum\limits_{k=k_{0}}^{M}\left(\frac{1}{n_{k}}-\frac{1}{n_{k+1}}\right)\right) \\
& \leq 2\|f\|_{\infty} \frac{1}{|I|} \sum\limits_{x \in I} \sum\limits_{y \in I} \frac{1}{n_{k_{0}}} \\
& \leq 2\|f\|_{\infty} \frac{1}{|I|} \sum\limits_{x \in I} \sum\limits_{y \in I} \frac{1}{|I|} \\
& \leq 2\|f\|_{\infty} .
\end{aligned}
$$

We next estimate $B_{4}$. We have

$$
\begin{aligned}
& B_{4}=\frac{1}{|I|^{2}} \sum\limits_{x \in I} \sum\limits_{y \in I}\left(\sum\limits_{k=k_{0}}^{M}\left|\frac{1}{n_{k+1}}\left(\sum\limits_{t=x+n_{k}+1}^{x+n_{k+1}} f_{2}(t)-\sum\limits_{t=y+n_{k}+1}^{y+n_{k+1}} f_{2}(t)\right)\right|\right) \\
& \leq \frac{1}{|I|^{2}} \sum\limits_{x \in I} \sum\limits_{y \in I}\left(\sum\limits_{\substack{k=k_{0} \\
n_{k+1}-n_{k}>|I|}}^{M}\left|\frac{1}{n_{k+1}}\left(\sum\limits_{t=x+n_{k}+1}^{x+n_{k+1}} f_{2}(t)-\sum\limits_{t=y+n_{k}+1}^{y+n_{k+1}} f_{2}(t)\right)\right|\right)+ \\
& \quad+\frac{1}{|I|^{2}} \sum\limits_{x \in I} \sum\limits_{y \in I}\left(\sum\limits_{\substack{k=k_{0} \\
n_{k+1}-n_{k} \leq|I|}}^{M}\left|\frac{1}{n_{k+1}}\left(\sum\limits_{t=x+n_{k}+1}^{x+n_{k+1}} f_{2}(t)-\sum\limits_{t=y+n_{k}+1}^{y+n_{k+1}} f_{2}(t)\right)\right|\right) \\
& =B_{4}(1)+B_{4}(2) .
\end{aligned}
$$

Now $B_{4}(2)$ is zero because of our lacunarity assumption; when $n_{k}>|I|$, then $n_{k+1} \geq 2 n_{k}$ and so $n_{k+1}-n_{k} \geq n_{k}>|I|$ too. For $B_{4}(1)$ we estimate as follows. Again, consider the two cases that $x \geq y$ or $y \geq x$. In the first case, we still have $x+n+k+1 \geq y+n_{k+1}$ because $n_{k+1}-n_{k}>|I|$. So as with the estimate for $B_{3}$, the difference $\sum\limits_{t=x+n_{k}+1}^{x+n_{k+1}} f_{2}(t)-$ $\sum\limits_{t=y+n_{k}+1}^{y+n_{k+1}} f_{2}(t)$ is no more in absolute value than $2\|f\|_{\infty}|I|$. Therefore,

$$
\begin{aligned}
B_{4}(1) & \leq \frac{1}{|I|^{2}} \sum\limits_{x \in I} \sum\limits_{y \in I}\left(\sum\limits_{k=k_{0}}^{M} \frac{1}{n_{k+1}}\left(2\|f\|_{\infty}|I|\right)\right) \\
& \leq 2\|f\|_{\infty} \frac{1}{|I|} \sum\limits_{x \in I} \sum\limits_{y \in I}\left(\sum\limits_{k=k_{0}}^{M} \frac{1}{n_{k+1}}\right) \\
& \leq 2\|f\|_{\infty} \frac{1}{|I|} \sum\limits_{x \in I} \sum\limits_{y \in I} \frac{1}{n_{k_{0}}} \\
& \leq C_{\infty}\|f\|_{\infty},
\end{aligned}
$$

\noindent because of the lacunarity assumption again.

Putting the estimates above together gives a constant $C_{\infty}$ depending only on $C_{2}$ and the lacunarity constant of the sequence $\left(n_{k}\right)$ such that for all $f \in L_{\infty}(X)$,

$$
\|S f\|_{B M O} \leq\|f\|_{\infty} .
$$

As explained above, this gives a constant $C_{p}$, when $1<p<\infty$, which does not depend on $M$ or $\left(c_{k}\right)$ such that for all $f \in l_{p}(\mathbb{Z})$, we have

$$
\|S f\|_{p} \leq C_{p}\|f\|_{p}
$$

One then transfers this to the dynamical system to give the bound needed to guarantee the unconditional convergence of the series of differences in $L^p(X)$.
\end{proof}

\section{Averages for Rotations on the Circle}\label{7}

There are also interesting issues with absolute and unconditional convergence of series of the form $\sum\limits_{k=1}^\infty TM_{n_k}f$ in all of the cases of the operators $T_n$ considered above.  We can generally show that these series are generically not absolutely convergent on $L^{p}$ for all $p, 1\le p \le \infty$.  The same thing is true under the weaker condition that the series is unconditionally convergent.  But then there is not any obvious Baire category argument to show that for some, or even perhaps all, increasing $(n_k)$ there would be some function $f\in L^p$ such that $\sum\limits_{k=1}^\infty M_{n_k}f$ is unconditionally convergent but not absolutely convergent.  

Instead of considering this as above in generality, we take another model to illustrate the issues: the averages by an ergodic rotation on the circle.  For $f\in C(\mathbb T)$ and $\beta \in \mathbb T$, denote by $f\circ \beta$ the function $f\circ \beta(\gamma) = f(\beta\gamma)$.  This model can be studied without having to deal with the technicalities of measure theory that come in when working with the Lebesgue spaces $L^p$. 

Let

\[M_n^\alpha f(\gamma) = \frac 1n\sum\limits_{j=1} ^n f\circ \alpha^j(\gamma)\] 
 
\noindent with mean zero $f\in C(\mathbb T)$.  By an ergodic rotation we mean an $\alpha$ which is not a root of unity.

We want to consider the behavior of the series $\mathcal Sf = \sum\limits_{k=1}^\infty M_{n_k}^\alpha f$ with $f\in C(\mathbb T)$,  Here we assume that $n_k < n_{k+1}$ for $k\ge 1$.  These types of series are interesting in that for a given mean-zero $f\in C(\mathbb T)$, we  can use the uniform convergence of $M_n^\alpha f$ to $\int f = 0$ to speed up the $n_k$ and get absolute convergence of the series $Sf$.  

We want to prove both Proposition~\ref{rotnotas} and Proposition~\ref{rotnotunc}.  We can use {\em almost invariant functions} to prove these results.

\begin{prop}\label{rotnotas} For all increasing $(n_k)$, there is a dense $G_\delta$ set $\mathcal B$ in $C(\mathbb T)$ such that for any $f \in \mathcal B$,   the series $\mathcal Sf$ is not absolutely convergent.   
\end{prop}
\begin{proof}
Consider the functions $f\in C(\mathbb T)$ such that $\mathcal Sf = \sum\limits_{k=1}^\infty M_{n_k}^\alpha f$ converges absolutely.  Denote this class of functions by $\mathcal A$.  These functions can be described as those for which there exists $C$ such that
$\sum\limits_{k=1}^\infty \|M_{n_k}^\alpha f\|_\infty \le C$.  
That is, for all $K\ge 1$, 
$\sum\limits_{k=1}^K \|M_{n_k}^\alpha f\|_\infty \le C$.  Call such a set $\mathcal A_C$. It is easy to see that each $\mathcal A_C$ is a closed set in the uniform norm.  Indeed, if $f_s$ are continuous functions converging uniformly to $f$, then for each $n$, $M_n^\alpha f_s$ converges uniformly to $M_n^\alpha f$.  So if we have a bound $\sum\limits_{k=1}^K \|M_{n_k}^\alpha f_s\|_\infty \le C$ for all $K$ and $s$, then the same thing would be true for $f$ i.e. for all $K$,  
$\sum\limits_{k=1}^K \|M_{n_k}^\alpha f\|_\infty \le C$.  

What we have is that $\mathcal A = \bigcup\limits_{C=1}^\infty \mathcal A_C$, and each $\mathcal A_C$ is a closed set.  If we can show that each $\mathcal A_C$ has no interior, then we will have $\mathcal A$ is a set of first category, and so its complement, the functions for which $\mathcal Sf$ is not absolutely convergent, would be generic - a dense $G_\delta$ set.  

Now why do the sets $\mathcal A_C$ have no interior?  If this were not the case, then there exists $f_0$ and $0 < \delta < 1$ so that for all $\|f\|_\infty \le 1$, we have $\delta f + f_0\in \mathcal A_C$.  In particular, $f_0\in \mathcal A_C$ because we can let $f =0$.  Also, then for any $f$ with $\|f\|_\infty \le 1$,  
\[\sum\limits_{k=1}^\infty \|M_{n_k}^\alpha f\|_\infty = \sum\limits_{k=1}^\infty \frac 1\delta \|M_{n_k}^\alpha \delta f + f_0 - f_0\|_\infty.\]
So by the triangle inequality, we have 
\[\sum\limits_{k=1}^\infty \|M_{n_k}^\alpha f\|_\infty \le \frac 1\delta \left (\sum\limits_{k=1}^\infty \|M_{n_k}^\alpha (\delta f+f_0))\|_\infty +                                   \sum\limits_{k=1}^\infty 
\|M_{n_k}^\alpha f_0\|_\infty \right )\le \frac {2C}\delta.\]   
But for any $K$, the construction using almost invariant functions below would actually show how to take some large $L$, and then a suitable almost invariant function $f$ of norm one, and get $\sum\limits_{k=1}^\infty \|M_{n_k}^\alpha f\|_\infty   \ge L$.  Once $L > \frac {2C}\delta$, this is a contradiction.  So each 
$\mathcal A_C$ has no non-trivial interior.   

Now we need the details of how almost invariant functions work in this context.  We start with the assumption that some $\mathcal A_C$ has non-trivial interior and conclude as above that for some $\delta$ we would have for all $f$ with $\|f\|_\infty \le 1$, $\sum\limits_{k=1}^\infty \|M_{n_k}^\alpha f\|_\infty \le \frac {2C}\delta$.  Now we are going to use almost invariant functions to show this cannot hold.

First, say you want the almost invariance for $f\in C(T)$ of norm one to happen so that for all $n_k, k=1,2,3,\dots,K$, $M_{n_k}^\alpha f$ is suitably close to $f$.  To do this we take disjoint closed arcs $C_1,C_2$ of the same length with translations $\alpha^l(C_1\cup C_2)$ pairwise disjoint for all $l=-2L,\dots,2L$.   Take isometric triangular tent functions $f_i\ge 0$, non-zero only on $C_i$, with base the whole interval $C_i$ and with $f_i = 1$ in the middle of $C_i, i=1,2$.  So these functions are also zero on all the images $\alpha^l(C_1\cup C_2)$.  We have $\int f_1 = \int f_2$.  

Take scalars $s_l > 0, l=1,\dots,2L$ that start small, increase to $1$, and then decrease to a small value, as $l$ increases.  If $L$ is large, we can do this linearly up and then down with a small change in value.  For example, take  $(s_1,\dots,s_{2L}) = (\frac 1L, \frac 2L,\dots,\frac LL,\frac LL,\dots,\frac 2L,\frac 1L)$.  Then take 

\[f = s_1f_1\circ \alpha^{-1} + s_2f_1\circ \alpha^{-2}+\dots+s_{2L-1}f_1\circ \alpha^{-(2L-1)} + s_{2L}f_1\circ \alpha^{-2L} \]
\[\large {-}	\left (s_1f_2\circ \alpha^{-1} + s_2f_2\circ \alpha^{-2}+\dots+s_{2L-1}f_2\circ \alpha^{-(2L-1)} + s_{2L}f_2\circ \alpha^{-2L}\right ).\]
	
We have chosen $f_1$ and $f_2$ so that all the terms above are pairwise disjointly supported.  Now $f$ is a sum of two parts, a positive part that is a string of small tent functions that go up in height (to more positive values with a maximum value of $1$) and then down in height back toward zero, and a negative part that is a string of small tent functions that go down in height (to more negative values with a minimum value of $-1$) and then up in height back toward zero.  This is how $f$ can be mean zero and uniform norm one.

Now when we shift by powers of $\alpha$, we will get $\|f\circ \alpha^j – f\|_\infty$ small for all $j$ with $j\le J$, and $J$ large, if we have made the step wise increase and then decrease in the $s_l$ very small (again needing $L$ very large).  In particular, for $k  \le K$, we have
\[\|\frac 1{n_k}\sum\limits_{j=1}^{n_k} f\circ \alpha^j - f\|_\infty =
\|\frac 1{n_k}\sum\limits_{j=1}^{n_k} \left (f\circ \alpha^k - f\right )\|_\infty
\le \frac 1{n_k} \sum\limits_{j=1}^{n_k} \|f\circ \alpha^j - f\|_\infty\]
\[\le \frac 1{n_k}\sum\limits_{j=1}^{n_k} \frac jL \le \frac {n_K}L.\]
Here the overestimate $\|f\circ \alpha^j - f\|_\infty\le \frac jL$ holds because of the choice of the values of $s_l$ to be changing from term to term by a value $1/L$.

Now we can begin the estimates with fixed $K \ge 8\frac {2C}\delta$.  Then we can take $L$ so large that $\frac {n_K}L \le \frac 1{10}$.  We use the triangle inequality
\[\|f\|_\infty = \|-f\|_\infty\le \|\frac 1{n_k}\sum\limits_{j=1}^{n_k} f\circ \alpha^j -f\|_\infty + \|\frac 1{n_k}\sum\limits_{j=1}^{n_k} f\circ \alpha^k \|_\infty.\]
Hence
\[\frac {2C}\delta \ge \sum\limits_{k=1}^K \|M_{n_k}^\alpha f\|_\infty \ge
\sum\limits_{k=1}^K \left (\|f\|_\infty - \|\frac 1{n_k}\sum\limits_{j=1}^{n_k} f\circ \alpha^j - f\|_\infty\right )\]
\[ \ge K\|f\|_\infty- K \frac {n_K}L\ge K(1-\frac {n_K}L)\ge \frac 9{10}K \ge \frac 9{10}\left (8\frac {2C}\delta\right ) = \frac {144}{10} \left (\frac C\delta\right ).\]
But then $2 \ge \frac {144}{10}$, which it is not.  
The conclusion is that no $\mathcal A_C$ can have non-trivial interior.
\end{proof}
                    
\begin{prop}\label{rotnotunc} For all increasing $(n_k)$, there is a dense $G_\delta$ set $\mathcal B$ in $C(\mathbb T)$ such that for any $f \in \mathcal B$,   the series $\mathcal Sf$ is not absolutely convergent.
\end{prop}
\begin{proof}
Similarly to the previous proof, we claim that the functions for which $\mathcal Sf$ is not unconditionally convergent comprise a generic set, formally a smaller class than the ones which are not absolutely convergent.  We do this by using the  result.  Denote the functions $f$ for which $\mathcal Sf$ is unconditionally convergent by $\mathcal U$.   B the result in Sierpinski~\cite{Serp}, or as redone in Heil~\cite{Heil}, these functions are such that there exists $C$ such that 
$\|\sum\limits_{m=1}^M \left |M_{n_m}^\alpha f\right | \|_\infty \le C$ for all $M\ge 1$.  Call this class $U_C$.  Then $\mathcal U = \bigcup\limits_{C=1}^\infty \mathcal U_C$.  By the same type of argument used above, we can show that $\mathcal U_C$ is a closed set with no interior in the uniform topology of $C(\mathbb T)$.  
\end{proof}
 
It would be a surprising result if $\mathcal Sf$ and $\mathcal Cf$ were the same.
\medskip

\noindent {\bf Conjecture}: For some (or perhaps all) increasing $(n_k)$  and some $f\in C(\mathbb T)$ we can have $Sf$ is unconditionally convergent but not absolutely convergent.
\bigskip

\begin{rem}  A {\em coboundary} for an ergodic rotation by $\alpha$ is a function $f\in C(\mathbb T)$ such that $f = F- F\circ \alpha$ for some $F\in C(\mathbb T)$.  The computation of $M_{n_k}^\alpha f$ is particular easy in this case.  Because of the telescoping of the sum that occurs, $M_{n_k}^\alpha f = \frac 1{n_k}(F\circ \alpha - F\circ \alpha^{n_k+1})$.  So if $1/n_k$ is summable, e.g. if $n_k = k^2$ for all $mk$, then $f = F - F\circ \alpha$ would be in $\mathcal A$.

In addition, it is a standard argument to show that the coboundaries under rotation by $\alpha$ are dense in $C(\mathbb T)$.  Hence, when $1/n_k$ is summable, the class $\mathcal A$ is first Baire category and dense in the uniform norm.  Of course, our argument that $\mathcal A$ is first category shows, among other things, the well-known fact is that for any ergodic rotation not every $f\in C(\mathbb T)$ is a coboundary.  
\qed
\end{rem}

\subsection{The Issue for the {\bf Conjecture}}
As we know, there are series $\mathcal S = \sum\limits_{m=1}^\infty f_m$ of continuous functions such that $\mathcal Sf$ is not absolutely convergent but it is unconditionally convergent.  We would like to construct examples of $Sf$ above for which the same thing happens.  It may be that this happens for all $S$ i.e. no matter what increasing sequence $n_k$ we choose.  But this is not clear. 

Suppose we have a sequence $(n_m)$.  Then for all $f\in C(\mathbb T)$, we have
\[\|\sum\limits _{m=1}^\infty \left |M_{n_m}^\alpha f\right | \|_\infty
\le \sum\limits _{m=1}^\infty \|M_{n_m}^\alpha f \|_\infty.\]
That is, $\mathcal A \subset \mathcal U$.  
If $\mathcal A$ and $\mathcal U$ are actually equal, it may be (but this is unclear) that then there is a constant $G$ such that for all $f\in C(\mathbb T)$,
\[\sum\limits _{m=1}^\infty \|M_{n_m}^\alpha f\|_\infty
\le G\|\sum\limits _{m=1}^\infty \left |M_{n_m}^\alpha f\right | \|_\infty.\]
This would mean that for any $M$,
\[\sum\limits _{m=1}^M \|M_{n_m}^\alpha f  \|_\infty
\le G\|\sum\limits _{m=1}^\infty \left |M_{n_m}^\alpha f\right | \|_\infty.\]

It might be possible for suitable $(n_m)$ to show that for any constant $G$ there are functions and some $M$ for which this inequality is not true.  It seems that what is needed is to get, for $m=1,\dots,M$, the norms $\|M_{n_m}^\alpha f\|_\infty$ are decreasing but not too quickly while the the functions $M_{n_m}^\alpha f$ are reasonably disjointly supported.  This approach might give us an indirect way of answering 3), essentially a Baire category approach to showing that $\mathcal U$ is not equal to $\mathcal A$.  

\subsection{Moving Averages}

There is a novel use of moving averages to relate absolute and unconditional convergence of series.  Suppose we have $(f_m)$ continuous on $\mathbb T$.  Then $\sum\limits_{k=1}^\infty \|f_k\|_\infty = \infty$ implies that there exists $v_k$ such that 
$\sum\limits_{k=1}^\infty |f_k\circ \alpha^{v_k}|$ is not bounded.  To see this just take $I_k = \{|f_k|\ge (1/2)\|f_k\|_\infty\}$ for all $m$.  These each contain non-trivial closed intervals.  So we can choose $v_k$ so that 
$\bigcap\limits_{k=1}^\infty \alpha^{-v_k}I_k \not= \emptyset$.  So actually $\sum\limits_{k=1}^\infty f_k$ is absolutely convergent if and only if ALL series of translations $\sum\limits_{k=1}^\infty f_k\circ \alpha^{v_k}$ are unconditionally convergent.  Of course, then also all the series 
$\sum\limits_{k=1}^\infty f_k\circ \alpha^{v_k}$ are absolutely convergent.   Without using all moving averages this would not hold.  There are lots of examples of this.  However, there is a trivial version in that there are series of continuous functions which are not absolutely convergent, but are unconditionally convergent.  And every series is a moving average with $v_k = 0$ for all $k$.

We can use moving averages to take advantage of the above remark.  A moving  average in this situation is a translation of the averages $M_n^\alpha f$.  To do this fix $v$ and let $M_{n,v}^\alpha f = \frac 1n\sum\limits_{j=v+1}^{v+n} f\circ \alpha^k$.  This is the same as $M_n^\alpha f \circ \alpha^v$.  Sometimes the behavior of moving averages is very different than the basic average without the translation, especially when dealing with differences between mean convergence and pointiwse convergence.  But in this case, if $f\in C(\mathbb T)$ is mean zero, then for any $v_n$, the averages $M_{n,v_n}^\alpha f$ converge uniformly to zero.  This is because the translation is actually an isometry in the uniform norm.

Now we can apply the observations above about translations to moving averages.  Take the series $Sf = \sum\limits_{k=1}^\infty M_{n_k}^\alpha f$.   Of course, the series $Sf$ is 
absolutely convergent if and only if all moving average series 
$\mathcal Mf = \sum\limits_{k=1}^\infty M_{n_k}^\alpha\circ \alpha^{v_k} f$ are absolutely convergent.  But in fact the short argument above gives the stronger implication.

\begin{prop} For $f\in C(\mathbb T)$, in the uniform norm, the series $Sf = \sum\limits_{k=1}^\infty M_{n_k}^\alpha f$ is absolutely convergent if and only if all moving average series $\mathcal Mf = \sum\limits_{k=1}^\infty M_{n_k}^\alpha\circ \alpha^{v_k} f$ are unconditionally convergent.    
\end{prop}

\subsection{Absolute and Unconditional Convergence of Series of Differences of Averaging with Rotations}

We hope in the future to consider the series $D_{(n_k)}f = \sum\limits_{k=1}^\infty M_{n_{k+1}}^\alpha f - M_{n_k}^\alpha f$ for $f\in C(\mathbb T)$.  
\medskip

\noindent {\bf Question}: Is the class of $f\in C(\mathbb T)$ such $D_{(n_k)}f$ is absolutely convergent always of first Baire category?
\medskip

But nonetheless, we also ask:
\medskip

\noindent {\bf Question}: Does there exist $(n_k)$ such that $D_{(n_k)}f$ is unconditionally convergent for all $f\in C(\mathbb T)$?

\section{General Questions}\label{8}

Trivially, we can have a sequence of bounded linear operators $T_n:B_1\to B_2$ such that $Sf = \sum\limits_{n=1}^\infty T_nf$ is absolutely convergent for all $f\in B_1$.  We just take the operators to have $\sum\limits_{n=1}^\infty \|T_n\| < \infty$.  So then the series $Sf$ is always unconditionally convergent too.  
\medskip

\noindent {\bf Question 1)}: Is there a natural condition on the operators that guarantees $Sf $ is always unconditionally convergent, but also at the same time there are $f\in B_1$ such that the series $Sf$ is not absolutely convergent?
\medskip

\noindent {\bf Question 2)}: Is there a natural condition on the operators that guarantees there exists $f\in B_1$ such that $Sf $ is unconditionally convergent, but $Sf$ is not absolutely convergent?  Examples where generically unconditional convergence fails would be particularly worthwhile.
\medskip

\noindent {\bf Question 3}: Given the results in this article, it would be worthwhile to have more revealing geometric or analytical conditions that are equivalent to the series of differences of Fourier transforms.  For example, how can we better describe the condition that there is a uniform bound over all $\theta$ for $\sum\limits_{k=1}^\infty |m(e^{in_{k+1}\theta} - m(e^{in_k\theta})| $?  This would give us better insight into when the series $\sum\limits_{k=1}^\infty M_{n_{k+1}}f - M_{n_k}f$ converges unconditionally in $L^p$-norm for all $f\in L^p(X)$? 
\medskip

\noindent {\bf Question 4}:  Are there geometric properties of the series of differences like $\sum\limits_{k=1}^\infty M_{n_{k+1}}f - M_{n_k}f$ that are distinctly determined by having a uniform bound for $\sum\limits_{k=1}^\infty |m(e^{in_{k+1}\theta} - m(e^{in_k\theta})|^p $?  This article makes it clear what the answer is when $p=1$.  The extensive work on square functions
$\sum\limits_{k=1}^\infty \left |M_{n_{k+1}}f - M_{n_k}f\right |^2$ makes it clear what the answer is for $p=2$.  But what about when $1< p <\infty$?

\noindent{\bf Question 5}: Given different sequences $(n_k)$ or $\epsilon_k$ as appear in numerous places in this article, can we say more about the functions, and their class relationships, that have the associated differences converging unconditionally in norm?   When will one class be contained in the other?  When will this containment be a proper one?  And so on.
\medskip

\noindent {\bf Acknowledgement}:  The majority of the work in this article in Sections~\ref{2} through Section~\ref{6} comes from Johnson~\cite{Johnson}.  The work in Section~\ref{7} of this article was carried out as part of an Illinois Mathematics Lab research project in Fall 2024.  We thank the participants for their input in the development of the results: the student scholars were Konrad Freymiller, Michael Hunding, Rachel Li, Ian Wilson, Zhiyuan Yang, Shuwei Zhang and the graduate student mentor was Balint Gyenti.


\medskip

\begin{thebibliography}{99}

\bibitem{AJP} M. Akcoglu, R. Jones, and P. Schwartz, {\em Variation in probability, ergodic theory and analysis},
Illinois Jornal of Mathematics 42, no. 1 (1998) 154–177.

\bibitem{Argiris} G. Argiris, {\em Convergence in ergodic theory},
PhD thesis, University of Illinois at Urbana-Champaign, 2001, 51 pages.

\bibitem{BJR} A. Bellow, R. Jones, and J. Rosenblatt, {\em Convergence of moving averages}, Ergodic Theory and Dynamical Systems 10 (1990) 43-63.

\bibitem{BCP} A. Benedek, A. - P. Calderon, and R. Panzone, {\em Convolution operators on Banach space valued functions}, Proc. Nat. Acad. Sci. USA 48 (1962) 356-365.

\bibitem{BS} C. Bennett and R. Sharpley, {\em Interpolation of Operators}, Academic Press, Boston, 1988.

\bibitem{Bennett} G. Bennett, {\em Unconditional convergence and almost everywhere convergence}, Zeitschrift Wahrscheinlichkeitstheorie, 34 (1976) 135-155.

\bibitem{BP} C. Bessaga and A. Pelczy\'nski {\em On bases and unconditional convergence of series in Banach Spaces}, Studia Mathematica 17, issuee 2 (1958) 151-164.

\bibitem{Cartwright} M. L. Cartwright, {\em On certain integral functions of order $1$}, Quarterly Journal of Mathematics. (Oxford) Ser. 7 (1936) 46-65.

\bibitem{Day} M. M. Day, {\em Normed Linear Spaces}, Lecture Notes in Math, No. 21, Ergebnisse der Mathematik und ihrer Grenzgebiete, Springer-Verlag, Berlin-Heidelberg-New York, 1977.

\bibitem{Heil} C. Heil, {\em A Basis Theory Primary, Expanded Edition}, Birkh\"auser, Springer New York, 1998.

\bibitem{Johnson} B. Johnson, {\em Unconditional convergence of differences in ergodic theory}, PhD thesis, Ohio State University, 1997, 90 pages.

\bibitem{JKRW} R. Jones, R. Kaufman, J. Rosenblatt, and M. Wierdl, {\em Oscillation in ergodic theory}, Ergodic Theory and Dynamical Systems 18 (1998) 889-935.

\bibitem{JOR} R. Jones, I. Ostrovskii, and J. Rosenblatt, {\em Square functions in ergodic theory}, Ergodic Theory and Dynamical Systems 16 (1996) 267-305.

\bibitem{LT} J. Lindenstrauss, L. Tzafriri, {\em  Classical Banach Spaces I and II, Sequence Spaces and Function Spaces}, Springer-Verlag, Berlin-Heidelberg-New York, 1996.

\bibitem{NS} A. Nagel and E. Stein, {\em On certain maximal functions and approach regions}, Advances in Mathematics 54 (1984) 83-106.

\bibitem{O}  W. Orlicz, {\em {\"Uber} unbedingte Konvergence in Funktionenraumen I}, Studia Math, 4 (1933), pp. 33-37.

\bibitem{Riviere} N. Rivi\`ere, {\em Interpolation \`a la Marcinkiewicz}, Revista de la Uni\'on Matemática Argentina, 213 (1971) 363-377.
 
\bibitem{Rosenblatt} J. Rosenblatt, {\em Ergodic group actions} Arch. Math. 47 (1986), 263-269.

\bibitem{Serp} W. Sierpinski, {\em L'influence de l'order des termes sur la convergence uniforme d'une s{\'e}rie} (Polish, French summary), Spraw. Tow. Nauk.   Warsz.
(Comptes Rendus des S{\'e}ances de la Soci{\'e}t{\'e} Scientifique de Varsovie),
3 (1910), pp. 354-357.

\bibitem{Timan} A. F. Timan, {\em Theory of Approximation of Functions of a Real Variable}, Dover Publications, 1994.

\bibitem{Zygmund} A. Zygmund, {\em Trigonometric Series}, (2 vols., reprinted in 1 vol.), Cambridge Univ. Press, Cambridge, U.K., 1990.

\end{thebibliography}
\end{document}